\documentclass[10pt, a4paper, leqno]{amsart}

\usepackage[T1]{fontenc}

\usepackage[abbrev, backrefs]{amsrefs}
\usepackage{microtype}
\usepackage{enumerate}
\usepackage{lmodern}
\usepackage{paralist}

\newtheorem{theorem}{Theorem}
\newtheorem{proposition}{Proposition}[section]
\newtheorem{claim}{Claim}
\newtheorem{lemma}[proposition]{Lemma}
\newtheorem{corollary}[theorem]{Corollary}
\theoremstyle{definition}
\newtheorem{remark}{Remark}[section]
\numberwithin{equation}{section}

\newcommand{\N}{{\mathbb N}}
\newcommand{\R}{{\mathbb R}}

\newcommand{\abs}[1]{\lvert #1 \rvert}
\newcommand{\bigabs}[1]{\bigl\lvert #1 \bigr\rvert}
\newcommand{\Bigabs}[1]{\Bigl\lvert #1 \Bigr\rvert}
\newcommand{\weakto}{\rightharpoonup}

\allowdisplaybreaks

\title[Groundstates of nonlinear Choquard equations]{Groundstates of nonlinear Choquard equations: existence, qualitative properties and decay asymptotics}
\author{Vitaly Moroz}
\address{Swansea University\\ Department of Mathematics\\ Singleton Park\\
Swansea\\ SA2~8PP\\ Wales, United Kingdom}	
\email{V.Moroz@swansea.ac.uk}

\author{Jean Van Schaftingen}
\address{Universit\'e Catholique de Louvain\\ Institut de Recherche en Math\'ematique et Phy\-sique\\ Chemin du Cyclotron 2 bte L7.01.01\\ 1348 Louvain-la-Neuve \\ Belgium}
\email{Jean.VanSchaftingen@uclouvain.be}



\date{\today}

\begin{document}

\begin{abstract}
We consider a semilinear elliptic problem
\[
 - \Delta u + u = (I_\alpha \ast \abs{u}^p) \abs{u}^{p - 2} u\quad\text{in \(\R^N\),}
\]
where \(I_\alpha\) is a Riesz potential and \(p>1\).
This family of equations includes the Choquard or nonlinear Schr\"odinger--Newton equation.
For an optimal range of parameters we prove the existence of a positive groundstate solution
of the equation. We also establish regularity and positivity of the groundstates and prove that all positive
groundstates are radially symmetric and monotone decaying about some point. Finally, we derive
the decay asymptotics at infinity of the groundstates.
\end{abstract}

\maketitle

\tableofcontents

\section{Introduction}

Given \(p \in (1, \infty)\), \(N \in \N_*=\{1, 2, \dotsc\}\) and \(\alpha \in (0, N)\), we
consider the problem
\begin{equation}
\label{eqChoquard}
 \left\{
   \begin{aligned}
      - \Delta u + u & = (I_\alpha \ast \abs{u}^p) \abs{u}^{p - 2} u & & \text{in \(\R^N\)},\\
      u (x) & \to 0 & & \text{as \(\abs{x} \to \infty\)},
   \end{aligned}
 \right.
\end{equation}
for \(u : \R^N \to \R\), where \(I_\alpha : \R^N \to \R\) is the Riesz potential defined by
\[
  I_\alpha(x)=\frac{\Gamma(\tfrac{N-\alpha}{2})}
                   {\Gamma(\tfrac{\alpha}{2})\pi^{N/2}2^{\alpha} \abs{x}^{N-\alpha}},
\]
and \(\Gamma\) is the Gamma function, see \cite{Riesz}*{p.\thinspace 19}.

If \(u\) solves \eqref{eqChoquard} when \(\alpha = 2\) and \(N \ge 3\),
the pair \((u, v) = (u, I_\alpha \ast \abs{u}^p)\) satisfies the system
\[
 \left\{
   \begin{aligned}
      - \Delta u + u & = v \,\abs{u}^{p - 2} u & & \text{in \(\R^N\)},\\
      - \Delta v & = \abs{u}^p & & \text{in \(\R^N\)},\\
      u (x) & \to 0 & & \text{as \(\abs{x} \to \infty\)},\\
      v (x) & \to 0 & & \text{as \(\abs{x} \to \infty\)}.
   \end{aligned}
 \right.
\]

Equation \eqref{eqChoquard} is usually called the nonlinear Choquard or Choquard--Pekar equation.
It has several physical origins.
In the physical case \(N = 3\), \(p=2\) and \(\alpha = 2\), the problem
\begin{equation}
\label{eqChoquard-physical}
 \left\{
   \begin{aligned}
      - \Delta u + u & = (I_2 \ast \abs{u}^2) u & & \text{in \(\R^3\)},\\
      u (x) & \to 0 & & \text{as \(\abs{x} \to \infty\)},
   \end{aligned}
 \right.
\end{equation}
appeared at least as early as in 1954, in a work by S.\thinspace I.\thinspace Pekar describing the quantum mechanics of a polaron at rest \cite{Pekar-1954}.
In 1976 P.\thinspace Choquard used \eqref{eqChoquard-physical} to describe an electron trapped in its own hole,
in a certain approximation to Hartree--Fock theory of one component plasma \cite{Lieb-1977}.
In 1996 R.\thinspace Penrose proposed \eqref{eqChoquard-physical} as a model of self-gravitating matter,
in a programme in which quantum state reduction is understood as a gravitational phenomenon \cite{Moroz-Penrose-Tod-1998}.
In this context equation of type \eqref{eqChoquard} is usually called the nonlinear Schr\"odin\-ger--Newton equation.
If \(u\) solves \eqref{eqChoquard} then the function \(\psi\) defined by \(\psi(t,x)=e^{it} u(x)\) is a solitary wave of the focusing time-dependent Hartree equation
\[
 i \psi_t = - \Delta \psi - (I_\alpha \ast \abs{\psi}^p) \abs{\psi}^{p-2}\psi\quad\text{in \(\R_+\times\R^N\)}.
\]
In this context \eqref{eqChoquard} is also known as the stationary nonlinear Hartree equation.

Problem \eqref{eqChoquard} has a variational structure: the critical points of the functional \(E_{\alpha, p} \in C^1 \bigl(W^{1, 2} (\R^N) \cap L^{\frac{2 N p}{N + \alpha}}(\R^N); \R\bigr)\) defined for \(u \in W^{1, 2} (\R^N) \cap L^{\frac{2 N p}{N + \alpha}}(\R^N)\) by
\begin{equation*}
  E_{\alpha, p} (u) = \frac{1}{2} \int_{\R^N} \abs{\nabla u}^2 + \abs{u}^2 - \frac{1}{2 p} \int_{\R^N} (I_\alpha \ast \abs{u}^p) \abs{u}^p
\end{equation*}
are weak solutions of \eqref{eqChoquard}. This functional is well defined by the Hardy--Littlewood--Sobolev inequality which states that if \(s \in (1, \frac{N}{\alpha})\) then for every \(v \in L^s (\R^N)\), \(I_\alpha \ast v\in L^\frac{N s}{N - \alpha s} (\R^N)\) and
\begin{equation}
\label{eqHLS}
 \int_{\R^N} \abs{I_\alpha \ast v}^\frac{N s}{N - \alpha s} \le C \Bigl(\int_{\R^N} \abs{v}^s \Bigr)^\frac{N}{N - \alpha s},
\end{equation}
where \(C>0\) depends only on \(\alpha\), \(N\) and \(s\).
Also note that by the Sobolev embedding, \(W^{1, 2} (\R^N)\subset L^{\frac{ 2 N p}{N + \alpha}}(\R^N)\) if and only if \(\frac{N - 2}{N + \alpha} \le \frac{1}{p} \le \frac{N}{N + \alpha}\).
\smallskip

The existence of solutions was proved for \(p = 2\) with variational methods by E.\thinspace H.\thinspace Lieb, P.-L.\thinspace Lions and G.\thinspace Menzala \citelist{\cite{Lieb-1977}\cite{Lions-1980}\cite{Menzala-1980}} and with ordinary differential equations techniques \citelist{\cite{Tod-Moroz-1999}\cite{Moroz-Penrose-Tod-1998}\cite{Choquard-Stubbe-Vuffray-2008}}.
We prove the following general existence result.

\begin{theorem}
\label{theoremExistence}
Let \(N \in \N_*\), \(\alpha \in (0, N)\) and \(p \in (1, \infty)\).
If \(\frac{N - 2}{N + \alpha} < \frac{1}{p} < \frac{N}{N + \alpha}\),
then there exists \(u \in W^{1, 2} (\R^N)\) such that
\[
  - \Delta u + u = (I_\alpha \ast \abs{u}^p) \abs{u}^{p - 2} u
\]
weakly in \(\R^N\). Moreover,
\begin{multline*}
  E_{\alpha, p} (u)
    = \inf \Bigl\{ E_{\alpha, p} (v) : v \in W^{1, 2} (\R^N) \setminus \{0\} \\
\text{ and }  \int_{\R^N} \abs{\nabla v}^2 + \abs{v}^2 =  \int_{\R^N} (I_\alpha \ast \abs{v}^p) \abs{v}^p \Bigr\}.
\end{multline*}
\end{theorem}

This existence result was mentioned without proof in \cite{Ma-Zhao-2010}*{p.\thinspace 457}.
Here we provide a complete proof of theorem \ref{theoremExistence}, with an emphasis
on the case \(p<2\) and the condition \(p > 1 + \frac{\alpha}{N}\) which replaces
the standard superlinearity condition \(p > 1\) for similar local problems.

The existence result of theorem \ref{theoremExistence} is sharp,
in the sense that if \(\frac{1}{p} \le \frac{N - 2}{N + \alpha} \) or \(\frac{1}{p} \ge \frac{N}{N + \alpha}\), problem \eqref{eqChoquard} does not have any sufficiently regular nontrivial variational solution. In particular, when \(p = 2\) and \(\alpha \le N - 4\), \eqref{eqChoquard} has no nontrivial smooth variational solution. Theorem~\ref{theoremExistence} can be extended to a general class of nonhomogeneous nonlinearities \cite{MorozVanSchaftingenGGCE}.

\begin{theorem}
\label{theoremPohozaev}
Let \(N \in \N_*\), \(\alpha \in (0, N)\) and \(p \in (1, \infty)\). Assume that \(u \in W^{1, 2} (\R^N) \cap L^{\frac{2 N p}{N + \alpha}}(\R^N)\) and that \(\nabla u \in W^{1, 2}_{\mathrm{loc}} (\R^N) \cap L^{\frac{2 N p}{N + \alpha}}_{\mathrm{loc}}(\R^N)\). If \(\frac{1}{p} \le \frac{N - 2}{N + \alpha}\) or \(\frac{1}{p} \ge \frac{N}{N + \alpha}\) and
\[
 - \Delta u + u = (I_\alpha \ast \abs{u}^p) \abs{u}^{p - 2} u
\]
weakly in \(\R^N\), then \(u = 0\).
\end{theorem}

The proof of theorem~\ref{theoremPohozaev} is based on a Poho\v zaev identity for \eqref{eqChoquard}.

\smallskip

We call \emph{groundstate} a function that satisfies the conclusions of theorem \ref{theoremExistence}.
When \(N = 3\), \(\alpha = 2\) and \(p = 2\), E.\thinspace H.\thinspace Lieb has proved that the groundstate \(u\) is radial and unique up to translations \cite{Lieb-1977}.
Wei Juncheng and M.\thinspace Winter have shown that the groundstate is up to translations a nondegenerate critical point \cite{Wei-Winter-2009}.

The \emph{symmetry} and the \emph{regularity} of solutions have been proved under the assumption
\begin{align}
\label{conditionMaZhao}
p &\ge 2 &
& \text{ and } &
 [2, \tfrac{2 N}{N - 2}]
  \cap (p, \tfrac{p N}{\alpha})
  \cap (\tfrac{(2 p - 2) N}{\alpha + 2}, \tfrac{(2 p - 1) N}{\alpha + 2})
  \cap [\tfrac{(2 p - 1) N}{N + \alpha}, \infty)
 &\ne \emptyset
\end{align}
respectively by Ma Li and Zhao Lin \cite{Ma-Zhao-2010} and by S.\thinspace Cingolani, M.\thinspace Clapp and S.\thinspace Secchi \cite{CingolaniClappSecchi2011}.
We prove here the regularity, positivity and radial symmetry of groundstates of Choquard equation \eqref{eqChoquard} for the optimal range of parameters.

\begin{theorem}
\label{theoremQualitative}
Let \(N \in \N_*\), \(\alpha \in (0, N)\) and \(p \in (1, \infty)\). Assume that \(\frac{N - 2}{N + \alpha} < \frac{1}{p} < \frac{N}{N + \alpha}\). If \(u \in W^{1, 2} (\R^N)\) is a groundstate of
\[
  - \Delta u + u = (I_\alpha \ast \abs{u}^p) \abs{u}^{p - 2} u\quad\text{in \(\R^N\)},
\]
then \(u \in L^1 (\R^N) \cap C^\infty (\R^N)\), \(u\) is either positive or negative and there exists \(x_0 \in \R^N\) and a monotone function \(v \in C^\infty (0, \infty)\) such that for every \(x \in \R^N\), \(u (x) = v (\abs{x - x_0})\).
\end{theorem}

In order to treat the radial symmetry of the solutions in the case \(p < 2\),
we use a variational argument based on the techniques of polarization \citelist{\cite{BartschWethWillem2005}\cite{VanSchaftingenWillem2008}}.

\smallskip

Finally we study the \emph{decay asymptotics} of the groundstates.
It was known that if \(p > 2\), then
\[
  u (x) = O (\abs{x}^{-\frac{N - 1}{2}} e^{-\abs{x}})
\]
as \(\abs{x} \to \infty\) and if \(p \ge 2\), then for every \(\epsilon > 0\),
\[
  u (x) = O (e^{-(1 - \epsilon) \abs{x}})
\]
as \(\abs{x} \to \infty\) \citelist{\cite{CingolaniClappSecchi2011}*{proposition A.2}\cite{Lions-1980}*{remark 1}}.
We have already studied optimal lower bounds of the asymptotics of \emph{distributional supersolutions} of Choquard equations in \emph{exterior domains} \cite{MorozVanSchaftingen}.
Here we derive the following sharp asymptotics for groundstates.

\begin{theorem}
\label{theoremAsymptotics}
Let \(N \in \N_*\), \(\alpha \in (0, N)\) and \(p \in (1, \infty)\). Assume that \(\frac{N - 2}{N + \alpha} < \frac{1}{p} < \frac{N}{N + \alpha}\). If \(u \in W^{1, 2} (\R^N)\) is a nonnegative groundstate of
\[
  - \Delta u + u = (I_\alpha \ast \abs{u}^p) \abs{u}^{p - 2} u\quad\text{in \(\R^N\)},
\]
then
\[
  \lim_{\abs{x} \to \infty} \frac{(I_\alpha \ast \abs{u}^p) (x)}{I_\alpha (x)} = \int_{\R^N} \abs{u}^p
\]
and
\begin{itemize}[--]
 \item if \(p > 2\),
\[
  \lim_{\abs{x} \to \infty} u (x) \abs{x}^{\frac{N - 1}{2}} e^{\abs{x}} \in (0, \infty);
\]
 \item if \(p = 2\),
\[
  \lim_{\abs{x} \to \infty} u (x) \abs{x}^{\frac{N - 1}{2}} \exp \int_{\nu}^{\abs{x}} \sqrt{1 - \frac{\nu^{N - \alpha}}{s^{N - \alpha}}} \,ds \in (0, \infty),
\]
where
\[
  \nu^{N - \alpha} = \frac{\Gamma(\tfrac{N-\alpha}{2})}{\Gamma(\tfrac{\alpha}{2})\pi^{N/2}2^{\alpha}} \int_{\R^N} \abs{u}^2
  = \frac{\Gamma(\tfrac{N-\alpha}{2})}{\Gamma(\tfrac{\alpha}{2})\pi^{N/2}2^{\alpha}}(\alpha + 4 - N) E_{\alpha, 2} (u);
\]
 \item if \(p < 2\),
\[
  \lim_{\abs{x} \to \infty} \bigl(u (x)\bigr)^{2 - p} \abs{x}^{N - \alpha} = \frac{\Gamma(\tfrac{N-\alpha}{2})}{\Gamma(\tfrac{\alpha}{2})\pi^{N/2}2^{\alpha}} \int_{\R^N} \abs{u}^p.
\]
\end{itemize}
\end{theorem}

In the case where either \(p > 2\) or \(p = 2\) and \(\alpha < N - 1\), similarly to the local stationary Schr\"odinger equation, the linear part dominates the nonlocal term and nonnegative groundstates of Choquard equation have the same asymptotics as the fundamental solution of the Schr\"odinger operator \(-\Delta+1\) on \(\R^N\) and thus \(u(x) \asymp \abs{x}^{-\frac{N - 1}{2}} e^{-\abs{x}}\).

If \(p < 2\) the decay of the groundstates is polynomial and entirely controlled by the nonlocal term in the equation.
A similar polynomial decay has been observed in the limiting integral equation for \(p = 1 + \frac{\alpha}{N}\) \cite{ChenLiOu2006}. Note that for the nonlinear Schr\"odinger equation with fractional Laplacian the solutions also decay
polynomially \cite{FelmerQuaasTan}*{theorem 1.3}; in the latter case however the polynomial decay is inherited from the linear operator whereas in our case it comes from the nonlinearity.
An interesting feature of the decay asymptotic is that it could be used to show that even though \(E_{\alpha, p}\) with \(p<2\) is nowhere twice Fr\'echet--differentiable, it is twice G\^ateaux--differentiable at a groundstate \(u\).
In particular, the question of nondegeneracy of \(u\) makes sense.

When \(p = 2\) and \(\alpha \ge N - 1\), the effect of the nonlocal term in the Choquard equation becomes strong enough in order to modify the asymptotics of the groundstates. All the three terms in the equation become balanced and each term should be taken into account in order to derive the asymptotics of the groundstates.
If \(p = 2\) and \(\alpha = N - 1\), the nonlocal term creates a polynomial correction to the standard exponential asymptotics,
so that
\(u(x)\asymp \abs{x}^{-\frac{N - 1}{2} + \frac{\nu^{N - \alpha}}{2}} e^{-\abs{x}}\).
If \(\alpha > N - 1\), the correction to the exponential asymptotics is an exponential
bounded by \(e^{c\abs{x}^{\alpha - (N - 1)}}\) (see remark~\ref{remarkLinearAsymptotics}).
When \(p=2\) the asymptotics of \(I_\alpha \ast \abs{u}^2\) depend only on the groundstate energy \(E_{\alpha, 2} (u)\); this might be useful in order to prove uniqueness of the groundstate in these cases.

In the classical physical case \(N=3\), \(\alpha=2\) and \(p = 2\), where the uniqueness of the groundstate is known \cite{Lieb-1977}, theorem~\ref{theoremAsymptotics} gives the decay asymptotics of the groundstate with an explicit polynomial correction
to the standard exponential decay rate \(\abs{x}^{-1} e^{-|x|}\) of the fundamental solution of \(-\Delta+1\) in \(\R^3\). This polynomial correction seems to be missing in \cite{Wei-Winter-2009}*{Theorem I.1 (1.7)}.

\begin{corollary}
If \(u \in W^{1, 2} (\R^N)\) is a nonnegative groundstate of
\[
  -\Delta u + u =(I_2 \ast \abs{u}^2)u\quad\text{in \(\;\R^3\)},
\]
then
\[
 \lim_{\abs{x} \to \infty} u(x) \abs{x}^{1-\frac{3}{8\pi} E_{2,2}(u)} e^{|x|}\in (0, \infty).
\]
\end{corollary}

\smallskip

The sequel of this paper is organized as follows.
In section \ref{sectionExistence} we explain the relations between different variational problems which
were historically used in connection with Choquard equation. Then we provide a detailed proof of theorem \ref{theoremExistence},
with an emphasis on including the case \(p < 2\), which was not considered before in the literature.
In section~\ref{sectionPohozaev}, we derive a Poho\v zaev identity for Choquard equation and prove the nonexistence result of theorem~\ref{theoremPohozaev}.
The four last sections are devoted to the proof of theorem~\ref{theoremQualitative}. We begin with the regularity (section~\ref{sectionRegularity}), continue with the positivity (section~\ref{sectionPositivity}) and the symmetry (section~\ref{sectionSymmetry}), and conclude with the asymptotics of the solutions in section~\ref{sectionAsymptotics}.

\section{Existence}\label{sectionExistence}

\subsection{Variational characterizations of the groundstate}

There are several ways to construct variationally a groundstate of \eqref{eqChoquard}.

The first way historically to construct solutions consists in minimizing the quantity
\[
  E^0_{\alpha, p} (u) = \frac{1}{2} \int_{\R^N} \abs{\nabla u}^2 - \frac{1}{2 p} \int_{\R^N} (I_\alpha \ast \abs{u}^p) \abs{u}^p,
\]
under the mass constraint \(\int_{\R^N} \abs{u}^2 = m\), in order to obtain a weak solution \(u \in W^{1, 2} (\R^N)\) of
\[
 - \Delta u + \mu u = \bigl(I_\alpha \ast \abs{u}^p\bigr) \abs{u}^{p - 2} u,
\]
where \(\mu \in \R\) is a Lagrange multiplier \citelist{\cite{Lieb-1977}\cite{Lions-1980}*{section 3}\cite{Lions1984CC1}*{section III}}. Such solutions are called \emph{normalized solutions}.
Another approach to the existence is to study the existence of a minimizer of \(E_{\alpha, p}\) under the Nehari constraint \(\langle E_{\alpha, p} '(u), u \rangle = 0\) \cite{Lions-1980}*{section 2}.
A third construction \cite{Lions1984CC1}*{section III.4} consists in minimizing the functional \(S_{\alpha, p} \in C^1 (W^{1, 2} (\R^N) \cap L^\frac{2 N p}{N + \alpha} (\R^N) \setminus \{0\})\) defined for \(u \in  W^{1, 2} (\R^N) \cap L^\frac{2 N p}{N + \alpha} (\R^N) \setminus \{0\}\) by
\[
 S_{\alpha, p} (u) = \frac{\int_{\R^N} \abs{\nabla u}^2+ \abs{u}^2}{\Bigl(\int_{\R^N} (I_\alpha \ast \abs{u}^p) \abs{u}^p\Bigr)^\frac{1}{p}}.
\]
One more approach to construct solutions variationally is to consider the functional associated to the interpolation inequality associated to the problem:
\[
 W_{\alpha, p} (u) = \frac{\Bigl(\int_{\R^N} \abs{\nabla u}^2\Bigr)^{\frac{N}{2} - \frac{N + \alpha}{2 p}} \Bigl(\int_{\R^N} \abs{u}^2\Bigr)^{\frac{N + \alpha}{2 p} - \frac{N - 2}{2}}}{\Bigl(\int_{\R^N} (I_\alpha \ast \abs{u}^p) \abs{u}^p\Bigr)^\frac{1}{p}},
\]
see \cite{Weinstein1982} for the local analogue.

The next proposition describes the relationships between the different minimization problems.

\begin{proposition}
\label{propositionVariational}
Let \(N \in \N_*\), \(\alpha \in (0, N)\) and \(p \in (1, \infty)\).
Let \(u \in W^{1, 2} (\R^N) \cap L^{\frac{2N p}{N + \alpha}} \setminus \{0\}\).
For \(t \in \R_*\), let \(u_t : \R^N \to \R\) be the dilated function defined for \(x \in \R^N\) by \(u_t (x) = u (t x)\).
One has
\[
 \max_{t > 0} E_{\alpha, p} (t u) = \Bigl(\frac{1}{2} - \frac{1}{2 p}\Bigr) S_{\alpha, p} (u)^\frac{p}{p - 1}.
\]
If \(\frac{N - 2}{N + \alpha} < \frac{1}{p} < \frac{N}{N + \alpha}\), then
\[
 \min_{t > 0} S_{\alpha, p} (u_t) =  \frac{2}{\frac{N + \alpha}{p} - (N - 2)} \biggl( \frac{\frac{1}{p} - \frac{N - 2}{N + \alpha}}{\frac{N}{N + \alpha} - \frac{1}{p}} \biggr)^{\frac{N}{2} - \frac{N + \alpha}{2 p}}
W_{\alpha, p} (u).
\]
If \(\frac{N}{N + \alpha + 2} < \frac{1}{p} < \frac{N}{N + \alpha}\), then
\begin{multline*}
  \min_{t > 0} E^0_{\alpha, p} (t^\frac{N}{2} u_t)
    = -\Bigl(\frac{1}{N (p - 1)-\alpha} - \frac{1}{2} \Bigr)\Bigl( \frac{2 p}{N (p - 1)-\alpha} \Bigr)^\frac{2}{N (p - 1)-2-\alpha}
\\ \Biggl(\frac{W_{\alpha, p} (u)}{\Bigl(\int_{\R^N} \abs{u}^2 \Bigr)^{\frac{N + \alpha}{2 p} - \frac{N - 2}{2}}}\Biggr)^\frac{2 p}{N (p - 1)-2-\alpha}.
\end{multline*}
If \(\frac{1}{p} = \frac{N}{N + \alpha + 2}\), then
\(E^0_{\alpha, p} (t^\frac{N}{2} u_t)=t^2 E^0_{\alpha, p} (u)\).\\
If \(\frac{1}{p} < \frac{N}{N + \alpha + 2}\), then \(\inf_{t > 0} E^0_{\alpha, p} (t^\frac{N}{2} u_t)=-\infty\) and
\begin{multline*}
  \max_{t > 0} E^0_{\alpha, p} (t^\frac{N}{2} u_t)
    = \Bigl(\frac{1}{2} - \frac{1}{N (p - 1)-\alpha} \Bigr)\Bigl( \frac{2 p}{N (p - 1)-\alpha} \Bigr)^\frac{2}{N (p - 1)-2-\alpha}
\\ \Biggl(\frac{W_{\alpha, p} (u)}{\Bigl(\int_{\R^N} \abs{u}^2 \Bigr)^{\frac{N + \alpha}{2 p} - \frac{N - 2}{2}}}\Biggr)^\frac{2 p}{N (p - 1)-2-\alpha}.
\end{multline*}
\end{proposition}

\begin{proof}
This is proved by a direct computation.
\end{proof}

This proposition shows that under some restrictions on the exponents, any infimum of one of the variational problems described
above can be written in terms of any other, and any minimizer of one problem is up to suitable dilation and rescaling a minimizer of the other problems. Note however that for \(\frac{1}{p} \le \frac{N}{N + \alpha + 2}\)
the energy \(E^0_{\alpha, p}(u)\) under the mass constraint \(\int_{\R^N} \abs{u}^2 = m\) is unbounded below and hence, as it was already remarked \citelist{\cite{Lions-1980}*{remark 9}\cite{Lions1984CC1}*{remark III.7}} in the case \(p = 2\), the minimization of \(S_{\alpha, p}\) works under less restrictions than the minimization of \(E^0_{\alpha, p}\) among normalized functions.

In this rest of this section, we shall prove the existence of a minimizer for the variational problems \(S_{\alpha,p}\).

\begin{proposition}
\label{propositionExistence}
Let \(N \in \N_*\), \(\alpha \in (0, N)\) and \(p \in (1, \infty)\). If  \(\frac{N - 2}{N + \alpha} < \frac{1}{p} < \frac{N}{N + \alpha}\), then there exists
\(
 u \in W^{1, 2} (\R^N)
\)
such that
\[
 S_{\alpha, p} (u) = \inf \bigl\{ S_{\alpha, p} (v) : v \in W^{1, 2} (\R^N) \text{ and } v \ne 0\bigr\}.
\]
\end{proposition}

Observe that theorem~\ref{theoremExistence} is a straightforward consequence of propositions~\ref{propositionVariational} and \ref{propositionExistence}.

There are several strategies to prove proposition~\ref{propositionExistence}. A first strategy consists in minimizing among radial functions and then prove with the symmetrization by rearrangement that a radial minimizer is a global minimizer. This approach was used for normalized solutions \citelist{\cite{Menzala-1980}\cite{Lieb-1977}}.

The other approach is the concentration-compactness method of P.-L.\thinspace Lions \citelist{\cite{Lions1984CC1}\cite{Lions1984CC2}}. In the sequel of this section we give a proof of proposition~\ref{propositionExistence} that relies on the simplest tools of concentration-compactness.

\subsection{Tools for the existence}

The first tool in our proof of proposition~\ref{propositionExistence} is a concentration-compactness lemma of P.-L. Lions that we reformulate  as an inequality \cite{Lions1984CC2}*{lemma I.1} (see also \cite{WillemMinimax}*{lemma 1.21}).

\begin{lemma}
\label{lemmaInequalitySup}
Let \(q \in [1, \infty)\). If \(\frac{1}{2} - \frac{1}{N} \le \frac{1}{q} \le \frac{1}{2}\)
then for every \(u \in W^{1, 2} (\R^N)\),
\[
 \int_{\R^N} \abs{u}^q \le C\Bigl( \sup_{x \in \R^N} \int_{B_1 (x)}  \abs{u}^q \Bigr)^{1 - \frac{2}{q}} \Bigl(\int_{\R^N}  \abs{\nabla u}^2 + \abs{u}^2\Bigr).
\]
\end{lemma}

\begin{proof}
By the Gagliardo--Nirenberg--Sobolev inequality on the ball \cite{AdamsFournier2003}*{theorem 5.8}, for every \(x \in \R^N\), one has
\[
   \int_{B_1(x)} \abs{u}^q
   \le C\Bigl( \sup_{x \in \R^N} \int_{B_1 (x)}  \abs{u}^q \Bigr)^{1 - \frac{q}{2}}
          \Bigl( \int_{B_1 (x)}  \abs{\nabla u}^2 + \abs{u}^2 \Bigr).
\]
We reach the conclusion by integrating \(x\) over \(\R^N\).
\end{proof}

Our second tool is a Brezis--Lieb lemma for the nonlocal term of the functional.

\begin{lemma}
\label{lemmaCompactnessDefaultRiesz}
Let \(N \in \N\), \(\alpha \in (0, N)\), \(p \in [1, \frac{2 N}{N + \alpha})\)
and \((u_n)_{n \in \N}\) be a bounded sequence in \(L^\frac{2 N p}{N + \alpha} (\R^N)\).
If \(u_n \to u\) almost everywhere on \(\R^N\) as \(n \to \infty\), then
\[
  \lim_{n \to \infty} \int_{\R^N} \bigl(I_\alpha \ast \abs{u_n}^p\bigr) \abs{u_n}^p
 - \int_{\R^N} \bigl(I_\alpha \ast \abs{u_n - u}^p\bigr) \abs{u_n - u}^p
  =  \int_{\R^N} \bigl(I_\alpha \ast \abs{u}^p\bigr) \abs{u}^p.
\]
\end{lemma}

In order to prove lemma~\ref{lemmaCompactnessDefaultRiesz}, we state an easy variant of the classical Brezis--Lieb lemma \cite{Brezis-Lieb-1983} (see also \citelist{\cite{Bogachev2007}*{proposition 4.7.30}\cite{Willem2013}*{theorem 4.2.7}}).

\begin{lemma}
\label{lemmaBrezisLieb}
Let \(\Omega \subseteq \R^N\) be a domain, \(q \in [1, \infty)\) and \((u_n)_{n \in \N}\)
be a bounded sequence in \(L^r (\Omega)\).
If \(u_n \to u\) almost everywhere on \(\Omega\) as \(n \to \infty\), then for every \(q \in [1, r]\),
\[
 \lim_{n \to \infty} \int_{\Omega} \bigabs{\abs{u_n}^q - \abs{u_n - u}^q - \abs{u}^q}^\frac{r}{q}
 = 0.
\]
\end{lemma}

Also recall that pointwise convergence of a bounded sequence implies weak convergence (see for example \citelist{\cite{Bogachev2007}*{proposition~4.7.12}\cite{Willem2013}*{proposition 5.4.7}}).

\begin{lemma}
\label{lemmaWeak}
Let \(\Omega \subseteq \R^N\) be a domain, \(q \in (1, \infty)\) and \((u_n)_{n \in \N}\) be a bounded sequence in \(L^q (\Omega)\).
If \(u_n \to u\) almost everywhere on \(\Omega\) as \(n \to \infty\), then \(u_n \weakto u\) weakly in \(L^q (\Omega)\).
\end{lemma}

We now have the ingredients to prove lemma~\ref{lemmaCompactnessDefaultRiesz}.

\begin{proof}[Proof of lemma~\ref{lemmaCompactnessDefaultRiesz}]
For every \(n \in \N\), one has
\begin{multline*}
\int_{\R^N} \bigl(I_\alpha \ast \abs{u_n}^p\bigr) \abs{u_n}^p
 - \int_{\R^N} \bigl(I_\alpha \ast \abs{u_n - u}^p\bigr) \abs{u_n - u}^p\\
  =  \int_{\R^N} \bigl(I_\alpha \ast (\abs{u_n}^p - \abs{u_n - u}^p)\bigr) (\abs{u_n}^p - \abs{u_n - u}^p)\\
+ 2 \int_{\R^N} \bigl(I_\alpha \ast (\abs{u_n}^p - \abs{u_n - u}^p)\bigr) \abs{u_n - u}^p.
\end{multline*}
By lemma~\ref{lemmaBrezisLieb} with \(q = p\) and \(r = \frac{2 N p}{N + \alpha}\), one has
\(\abs{u_n - u}^p - \abs{u_n}^p \to \abs{u}^p\),
strongly in \(L^\frac{2N}{N + \alpha} (\R^N)\) as \(n \to \infty\).
By the Hardy--Littlewood--Sobolev inequality \eqref{eqHLS}, this implies that
\(I_\alpha \ast (\abs{u_n - u}^p - \abs{u_n}^p) \to I_\alpha \ast \abs{u}^p\) in \(L^\frac{2N}{N - \alpha} (\R^N)\) as \(n \to \infty\). Since by lemma~\ref{lemmaWeak}, \(\abs{u_n - u}^p \weakto 0\) weakly in \(L^\frac{2N}{N + \alpha} (\R^N)\) as \(n \to \infty\), we reach the conclusion.
\end{proof}

\subsection{Proof of the existence}

We now employ lemmas~\ref{lemmaInequalitySup} and \ref{lemmaCompactnessDefaultRiesz} in order to prove proposition~\ref{propositionExistence}.

\begin{proof}[Proof of proposition~\ref{propositionExistence}]
Set
\[
\begin{split}
 c_{\alpha, p} &= \inf \Bigl\{ S_{\alpha, p} (u) : u \in W^{1, 2} (\R^N) \setminus \{0\}\Bigr\}\\
   &= \inf \Bigl \{ S_{\alpha, p} (u) : u \in W^{1, 2} (\R^N) \text{ and } \int_{\R^N} (I_\alpha \ast \abs{u}^p) \abs{u}^p = 1 \Bigr\}.
\end{split}
\]
Let \((u_n)_{n \in \N}\) be a sequence in \(W^{1, 2} (\R^N)\) such that
\(\int_{\R^N} (I_\alpha \ast \abs{u_n}^p) \abs{u_n}^p = 1\) for every \(n \in \N\),
and \(\lim_{n \to \infty} S_{\alpha, p} (u_n) = c_{\alpha, p}\).
In particular, the sequence \((u_n)_{n \in \N}\) is bounded in \(W^{1, 2} (\R^N)\).

By the Hardy--Littlewood--Sobolev inequality  \eqref{eqHLS} and by lemma~\ref{lemmaInequalitySup},
taking into account that \( \frac{1}{2} - \frac{1}{N}  \le \frac{N + \alpha}{2 N p} \le \frac{1}{2}\), we obtain
\[
\begin{split}
 \int_{\R^N} (I_\alpha \ast \abs{u_n}^p) \abs{u_n}^p
& \le C \Bigl( \int_{\R^N} \abs{u_n}^\frac{2 N p}{N + \alpha} \Bigr)^\frac{N + \alpha}{N}\\
& \le C' \biggl(\Bigl( \sup_{x \in \R^N}\int_{B_1 (x)}  \!\!\!\!\!\!\abs{u_n}^\frac{2 N p}{N + \alpha} \Bigr)^{1 - \frac{N + \alpha}{N p}}
 \Bigl(\int_{\R^N} \abs{\nabla u_n}^2 + \abs{u_n}^2\Bigr)\biggr)^\frac{N + \alpha}{N}.
\end{split}
\]
Since \((u_n)_{n \in \N}\) is bounded in \(W^{1, 2} (\R^N)\) and \(\frac{1}{p} < \frac{N}{N + \alpha}\), we deduce that the exists a sequence \((x_n)_{n \in \N}\) in \(\R^N\) such that
\begin{equation}
\label{eqExistenceInfBalls}
  \inf_{n \in \N} \int_{B_1 (x_n)} \abs{u_n}^\frac{2 N p}{N + \alpha} > 0.
\end{equation}
Since the problem is invariant by translation, we can assume that for every \(n \in \N\), \(x_n = 0\).
Because the sequence \((u_n)_{n \in \N}\) is bounded, we can also assume without loss of generality that \(u_n \weakto u \in W^{1, 2} (\R^N)\) weakly in \(W^{1, 2} (\R^N)\) as \(n \to \infty\).
By \eqref{eqExistenceInfBalls}, \(u \ne 0\). Note now that
\[
 \int_{\R^N} \abs{\nabla u}^2 + \abs{u}^2
 = \lim_{n \to \infty} \int_{\R^N} \abs{\nabla u_n}^2 + \abs{u_n}^2 - \int_{\R^N} \abs{\nabla (u_n - u)}^2 + \abs{u_n - u}^2,
\]
and therefore
\[
\begin{split}
 S_{\alpha, p} (u)
= & \lim_{n \to \infty} S_{\alpha, p} (u_n) \biggl( \frac{\int_{\R^N} \bigl(I_\alpha \ast \abs{u_n}^p\bigr)\abs{u_n}^p}{\int_{\R^N} \bigl(I_\alpha \ast \abs{u}^p\bigr)\abs{u}^p} \biggr)^\frac{1}{p}\\
&\qquad \qquad - S_{\alpha, p} (u_n - u) \biggl( \frac{\int_{\R^N} \bigl(I_\alpha \ast \abs{u_n - u}^p\bigr)\abs{u_n - u}^p}{\int_{\R^N} \bigl(I_\alpha \ast \abs{u}^p\bigr)\abs{u}^p} \biggr)^\frac{1}{p}\\
\le & c_{\alpha, p} \liminf_{n \to \infty} \biggl( \frac{\int_{\R^N} \bigl(I_\alpha \ast \abs{u_n}^p\bigr)\abs{u_n}^p}{\int_{\R^N} \bigl(I_\alpha \ast \abs{u}^p\bigr)\abs{u}^p} \biggr)^\frac{1}{p}\\
& \qquad \qquad \qquad - \biggl( \frac{\int_{\R^N} \bigl(I_\alpha \ast \abs{u_n - u}^p\bigr)\abs{u_n - u}^p}{\int_{\R^N} \bigl(I_\alpha \ast \abs{u}^p\bigr)\abs{u}^p} \biggr)^\frac{1}{p}.
\end{split}
\]
By lemma~\ref{lemmaCompactnessDefaultRiesz}, since \(p > 1\), we reach a contradiction if
\[
\int_{\R^N} \bigl(I_\alpha \ast \abs{u}^p\bigr)\abs{u}^p < \liminf_{n \to \infty} \int_{\R^N} \bigl(I_\alpha \ast \abs{u_n}^p\bigr)\abs{u_n}^p.
\]
We conclude that \(S_{\alpha, p} (u) = c_{\alpha, p}\) which completes the proof.
\end{proof}

\section{Poho\v zaev identity}
\label{sectionPohozaev}

In this section we prove the nonexistence result of theorem~\ref{theoremPohozaev} and establish
a relationships between the energy and \(L^2\)--norm of a solution.

\subsection{Poho\v zaev identity}

We establish the following Poho\v zaev type identity.

\begin{proposition}
\label{propositionPohozaev}
Let \(u \in W^{1, 2} (\R^N) \cap L^{\frac{2 N p}{N + \alpha}} (\R^N)\). If
\[
 - \Delta u + u = (I_\alpha \ast \abs{u}^p) \abs{u}^{p - 2} u.
\]
and \(u \in W^{2, 2}_\mathrm{loc}(\R^N) \cap W^{1,\frac{2 N}{N + \alpha} p} (\R^N)\), then
\[
  \frac{N - 2}{2} \int_{\R^N} \abs{\nabla u}^2 + \frac{N}{2} \int_{\R^N} \abs{u}^2
= \frac{N + \alpha}{2 p} \int_{\R^N} (I_\alpha \ast \abs{u}^p) \abs{u}^{p}.
\]
\end{proposition}

G.\thinspace Menzala has used a similar identity to prove the nonexistence of solutions to the equation \(- \Delta u + u - \beta I_\alpha (x) = (I_\alpha \ast \abs{u}^2) u\) \cite{Menzala-1983}. When \(N = 3\), \(\alpha = 2\) and \(p = 2\), proposition~\ref{propositionPohozaev} has been proved for smooth solutions \cite{CingolaniSecchiSquassina2010}*{lemma 2.1}.

Our proof of proposition~\ref{propositionPohozaev} follows a classical strategy of testing the equation against \(x \cdot \nabla u(x)\) which is made rigorous by multiplying by suitable cut-off functions \citelist{\cite{Kavian1993}*{proposition 6.2.1}\cite{WillemMinimax}*{appendix B}}.

\begin{proof}[Proof of proposition~\ref{propositionPohozaev}]
We take \(\varphi \in C^1_c(\R^N)\) such that \(\varphi = 1\) on \(B_1\). By testing the equation against the function \(v_\lambda \in W^{1, 2} (\R^N) \cap L^{\frac{2 N p}{N + \alpha}} (\R^N)\) defined for \(\lambda \in (0, \infty)\) and \(x \in \R^N\) by
\[
 v_\lambda (x) = \varphi(\lambda x)\, x \cdot \nabla u (x),
\]
we have
\[
   \int_{\R^N} \nabla u \cdot \nabla v_\lambda + \int_{\R^N} u v_\lambda
= \int_{\R^N} (I_\alpha \ast \abs{u}^p) \abs{u}^{p - 2} u v_\lambda.
\]
We compute for every \(\lambda > 0\),
\[
\begin{split}
 \int_{\R^N} u v_\lambda &= \int_{\R^N} u(x) \varphi (\lambda x)\, x \cdot \nabla u(x) \,d x\\
  &= \int_{\R^N} \varphi (\lambda x)\, x \cdot \nabla \bigl(\tfrac{\abs{u}^2}{2}\bigr) (x) \,dx\\
  &= - \int_{\R^N}  \bigl( N \varphi (\lambda x)
                         + \lambda x \cdot \nabla \varphi (\lambda x) \bigr)
                    \frac{\abs{u(x)}^2}{2}
        \,d x.
\end{split}
\]
By Lebesgue's dominated convergence theorem, it holds
\[
 \lim_{\lambda \to 0} \int_{\R^N} u v_\lambda = - \frac{N}{2} \int_{\R^N} \abs{u}^2.
\]
Next, since \(u \in W^{2, 2}_\mathrm{loc} (\R^N)\) we have
\[
\begin{split}
 \int_{\R^N} \nabla u \cdot \nabla v_\lambda
  &= \int_{\R^N} \varphi (\lambda x)
                  \bigl(\abs{\nabla u}^2
                        + x \cdot \nabla \bigl(\tfrac{\abs{\nabla u}^2}{2}\bigr) (x) \bigr)
      \,dx\\
  &= - \int_{\R^N}  \bigl( (N - 2) \varphi (\lambda x)
                         + \lambda x \cdot \nabla \varphi (\lambda x) \bigr)
                    \frac{\abs{\nabla u(x)}^2}{2}
        \,d x.
\end{split}
\]
By Lebesgue's dominated convergence again, since \(u \in W^{1, 2} (\R^N)\),
\[
 \lim_{\lambda \to 0} \int_{\R^N} \nabla u \cdot \nabla v_\lambda = - \frac{N - 2}{2} \int_{\R^N} \abs{\nabla u}^2.
\]
Finally, we have for every \(\lambda > 0\),
\[
\begin{split}
 \int_{\R^N} (I_\alpha \ast \abs{u}^p) &\abs{u}^{p - 2} u v_\lambda
 = \int_{\R^N} \int_{\R^N} \abs{u (y)}^p I_\alpha (x - y) \varphi (\lambda x)\, x \cdot \nabla \bigl(\tfrac{\abs{u}^p}{p}\bigr) (x) \,dx\,dy\\
 =& \frac{1}{2} \int_{\R^N} \int_{\R^N} I_\alpha (x - y) \Bigl(\abs{u (y)}^p \varphi (\lambda x) \,x \cdot \nabla \bigl(\tfrac{\abs{u}^p}{p}\bigr) (x)\\
  & \qquad \qquad \qquad \qquad \qquad + \abs{u (x)}^p \varphi (\lambda y) \,y \cdot \nabla \bigl(\tfrac{\abs{u}^p}{p}\bigr) (y)\Bigr)\,dx\,dy\\
 =& - \int_{\R^N} \int_{\R^N} \abs{u (y)}^p I_\alpha (x - y) \bigl(N \varphi (\lambda x) + x \cdot \nabla \varphi (\lambda x)\bigr) \frac{\abs{u (x)}^p}{p} \,dx\,dy\\
  &+ \frac{N - \alpha}{2 p} \int_{\R^N} \int_{\R^N} \abs{u (y)}^p I_\alpha (x - y)\\
  & \qquad \qquad \qquad \qquad  \qquad \frac{(x - y) \cdot \bigl(x \varphi (\lambda x) - y \varphi (\lambda y)\bigr)}{\abs{x - y}^2} \abs{u (x)}^p \,dx\,dy.
\end{split}
\]
We can thus apply Lebesgue's dominated convergence theorem to conclude that
\[
 \lim_{\lambda \to 0} \int_{\R^N} (I_\alpha \ast \abs{u}^p) \abs{u}^{p - 2} u\, v_\lambda
= - \frac{N + \alpha}{2 p} \int_{\R^N} (I_\alpha \ast \abs{u}^p) \abs{u}^{p}.\qedhere
\]
\end{proof}

\subsection{Proof of the nonexistence}

We now complete the proof of theorem~\ref{theoremPohozaev}.

\begin{proof}[Proof of theorem~\ref{theoremPohozaev}]
By testing the equation against \(u\), we obtain the identity
\[
 \int_{\R^N} \abs{\nabla u}^2 + \int_{\R^N} \abs{u}^2
= \int_{\R^N} (I_\alpha \ast \abs{u}^p) \abs{u}^{p}.
\]
Hence, we have
\[
  \Bigl(\frac{N - 2}{2} - \frac{N + \alpha}{2 p}\Bigr)\int_{\R^N} \abs{\nabla u}^2 + \Bigl(\frac{N}{2} - \frac{N + \alpha}{2 p}\Bigr) \int_{\R^N} \abs{u}^2 = 0.
\]
If \( \frac{1}{p} \le \frac{N - 2}{N + \alpha}\) or \(\frac{1}{p} \ge \frac{N}{N + \alpha}\), this implies that \(u = 0\).
\end{proof}

\subsection{An integral identity}
Proposition~\ref{propositionPohozaev} allows to express \(\int_{\R^N} \abs{u}^2\) in terms of \(E_{\alpha, p} (u)\).
We will see in section~\ref{sectionAsymptotics} that the asymptotics of the groundstates when \(p \le 1\) depend on \(\int_{\R^N} \abs{u}^p\).
When \(p = 2\), we will thus be able to express the asymptotics of a groundstate in terms of the energy of the groundstate.

\begin{proposition}
\label{propositionIntegralIdentity}
Let \(\alpha \in (0, N)\) and \(p \in (1, \infty)\) be such that \(\frac{N - 2}{N + \alpha} < \frac{1}{p} < \frac{N}{N + \alpha}\).
If \(u \in W^{1, 2}(\R^N)\) is a groundstate of
\[
 -\Delta u + u = (I_\alpha \ast \abs{u}^p) \abs{u}^{p - 2} u,
\]
then
\[
  \int_{\R^N} \abs{u}^2 = \Bigl(\frac{\alpha + 2}{p - 1} - (N - 2)\Bigr)E_{\alpha, p} (u).
\]
\end{proposition}

\begin{proof}
The proof is a direct computation.
\end{proof}

\section{Regularity}
\label{sectionRegularity}

In this section we study the regularity of a solution of \eqref{eqChoquard}.

\begin{proposition}
\label{propositionRegularity}
Let \(\alpha \in (0, N)\) and \(p \in (1, \infty)\) such that \(\frac{N - 2}{N + \alpha} < \frac{1}{p} < \frac{N}{N + \alpha}\).
Let \(u \in W^{1, 2}(\R^N)\) be a solution of
\[
 -\Delta u + u = (I_\alpha \ast \abs{u}^p) \abs{u}^{p - 2} u.
\]
Then \(u \in L^1 (\R^N) \cap C^2 (\R^N)\), \(u \in W^{2, s} (\R^N)\)  for every \(s > 1\), and \(u \in C^\infty \bigl(\R^N \setminus u^{-1} (\{0\})\bigr)\).
\end{proposition}

S.\thinspace Cingolani, M.\thinspace Clapp and S.\thinspace Secchi \cite{CingolaniClappSecchi2011}*{lemma A.1} have proved that \(u \in L^s (\R^N)\) for every \(s \ge 2\) and that \(u\) is smooth under the additional assumptions that \(p \ge 2\) and that
\eqref{conditionMaZhao} holds.

We shall split the proof of proposition \ref{propositionRegularity} in the proof of three separate claims.

\begin{claim}
\label{claimLow}
For every \(s \ge 1\) such that
\[
  \frac{1}{s} > \frac{\alpha}{N} \Bigl(1 - \frac{1}{p}\Bigr) - \frac{2}{N},
\]
one has \(u \in L^s (\R^N)\) and for every \(r > 1\) such that
\[
  \frac{1}{r} > \frac{\alpha}{N} \Bigl(1 - \frac{1}{p}\Bigr),
\]
one has \(u \in W^{2, r} (\R^N)\).
\end{claim}

\begin{proof}[Proof of claim~\ref{claimLow}]
Since \(u \in W^{1, 2} (\R^N)\), \(u \in L^{s_0} (\R^N)\) with
\[
 s_0 = \frac{2 N p}{N + \alpha}.
\]
Set \(\overline{s}_0 = \underline{s}_0 = s_0\).
Assume that \(u \in L^s (\R^N)\) for every \(s \in [\underline{s}_n, \overline{s}^n]\).
By the Hardy-Littlewood-Sobolev inequality \eqref{eqHLS},
if
\[\frac{1}{t} = \frac{p}{s} - \frac{\alpha}{N} > 0\]
then
\(
  I_\alpha \ast \abs{u}^p \in L^t (\R^N).
\)
Further, if
\[
 \frac{1}{r} = \frac{2 p - 1}{s} - \frac{\alpha}{N} < 1,
\]
then we have
\( (I_\alpha \ast \abs{u}^p) \abs{u}^{p - 2} u \in L^r (\R^N)\).
Hence, \(u \in W^{2, r}(\R^N)\) by the classical Calder\'on--Zygmund \(L^p\) regularity estimates  \cite{Gilbarg}*{chapter 9}.

Therefore, if
\[
 \frac{\alpha}{N p} < \frac{1}{s} < \frac{N + \alpha}{N (2 p - 1)},
\]
then
\(
  u \in W^{2, r}(\R^N).
\)
In other words, we have thus proved that \(u \in W^{2, r}(\R^N)\), for every \(r > 1\) such that
\[
  \frac{2 p - 1}{\overline{s}_n} - \frac{\alpha}{N} < \frac{1}{r} <  \frac{2 p - 1}{\underline{s}_n} - \frac{\alpha}{N}
\]
and
\[
 \frac{1}{r} > \frac{\alpha}{N} \Bigl(1 - \frac{1}{p}\Bigr).
\]
Finally, by the Sobolev embedding theorem, \(u \in L^s (\Omega)\) provided that
\[
   \frac{2 p - 1}{\overline{s}_n} - \frac{\alpha + 2}{N} < \frac{1}{s} <  \frac{2 p - 1}{\underline{s}_n} - \frac{\alpha}{N}
\]
and
\[
 \frac{1}{s} > \frac{\alpha}{N} \Bigl(1 - \frac{1}{p}\Bigr) - \frac{2}{N}.
\]
From \(\underline{s}_n < \frac{2 N p}{N + \alpha} < \overline{s}_n\), we deduce that
\[
  \frac{2 p - 1}{\overline{s}_n} - \frac{\alpha + 2}{N} < \frac{1}{\overline{s}_n}
  < \frac{2 N p}{N + \alpha} < \frac{1}{\underline{s}_n} < \frac{2 p - 1}{\underline{s}_n} - \frac{\alpha}{N}.
\]
If \( \frac{2 p - 1}{\underline{s}_n} - \frac{\alpha + 2}{N} \ge 1\), we are done. Otherwise, we set
\(\frac{1}{\underline{s}_{n + 1}}=\frac{2 p - 1}{\underline{s}_n} - \frac{\alpha}{N}\).
Similarly, if \( \frac{2 p - 1}{\overline{s}_n} - \frac{\alpha + 2}{N} \le \frac{\alpha}{N}(1 - \frac{1}{p}) - \frac{2}{N}\),
then we are done. Otherwise we set \(\frac{1}{\Bar{s}_{n + 1}} = \frac{2 p - 1}{\overline{s}_n} - \frac{\alpha + 2}{N} \)
and iterate. This allows to reach the conclusion after a finite number of steps.
\end{proof}

\begin{claim}
\label{claimHigh}
For every \(r > 1\), \(u \in W^{2, r} (\R^N)\) .
\end{claim}

The proof of this claim is classical, we give a proof in the same style as the proof of claim~\ref{claimLow} for the sake of completeness.

\begin{proof}[Proof of claim~\ref{claimHigh}]
Since \(\frac{1}{p} > \frac{N-2}{N + \alpha}\), we have \(\frac{\alpha}{N} (1 - \frac{1}{p}) - \frac{2}{N} < \frac{\alpha}{N p}\) and by claim~\ref{claimLow}, \(I_\alpha \ast \abs{u}^p \in L^\infty (\R^N)\).
If \(r_0 \in (1, \infty)\) is defined by \(\frac{1}{r_0} = \frac{\alpha}{N}(1 - \frac{1}{p})\), by claim~\ref{claimLow}, \(u \in W^{2, r} (\R^N)\) for every \(r \in (1, r_0)\).

Assume now that \(u \in W^{2, r} (\R^N)\) for every \(r \in (1, r_n)\).
By the classical Sobolev embedding theorem,
\(
  u \in L^s (\R^N)
\)
for every \(s \in [1, \infty)\) such that \(\frac{1}{s} > \frac{1}{r_n} - \frac{2}{N}\).
Hence, \((I_\alpha \ast \abs{u}^p) \abs{u}^{p - 2} u \in L^{r} (\R^N)\) for every \(r \in (1, \infty)\) such that \(p - 1 > \frac{1}{r} > (p - 1)(\frac{1}{r_n} - \frac{2}{N})\).
By the classical Calder\'on-Zygmund theory \cite{Gilbarg}*{chapter 9}, \(u \in W^{2, r} (\R^N)\) for every \(r \in (1, \infty)\) such that \(\frac{1}{r} > (p - 1)(\frac{1}{r_n} - \frac{2}{N})\).
If \(r_n \ge \frac{N}{2}\), we are done. Otherwise set
\[
  \frac{1}{r_{n + 1}} = (p - 1)\Bigl(\frac{1}{r_n} - \frac{2}{N}\Bigr).
\]
If \(p \le 2\), then
\[
 \frac{1}{r_{n + 1}} < \frac{1}{r_n} - \frac{2 (p - 1)}{N},
\]
and the conclusion is reached after a finite number of steps.
If \(p > 2\), then since \(\frac{1}{r_n} < \frac{\alpha}{N}(1 - \frac{1}{p})\) and \(\frac{1}{p}
\ge \frac{N - 2}{N + \alpha}\),
\[
\begin{split}
 \frac{1}{r_{n + 1}} & < \frac{1}{r_n} + (p - 1)\biggl(\frac{\alpha}{N}\Bigr(1 - \frac{2}{p}\Bigr) - \frac{2}{N}\biggr)\\
& < \frac{1}{r_n} + (p - 1)\biggl(\frac{\alpha}{N}\Bigr(1 - 2 \frac{N - 2}{N + \alpha}\Bigr) - \frac{2}{N}\biggr)
< \frac{1}{r_n};
\end{split}
\]
the conclusion is again reached after a finite number of steps.
\end{proof}

\begin{claim}
\label{claimHolder}
For every \(k \in \N\), \(\lambda \in (0, 1)\) and every ball \(B \subset \R^N\), if \(k + \lambda < p + 1\) or \(\inf_B \abs{u} > 0\), \(u \in C^{k, \lambda}_{\mathrm{loc}} (\R^N)\).
\end{claim}
\begin{proof}[Proof of claim~\ref{claimHolder}]
By claim~\ref{claimHigh} and by the Morrey--Sobo\-lev embedding, the conclusion holds for \(k \in \{0, 1\}\).
Let \(\Tilde{B} \subset \R^N\) be a ball such that \(\Bar{B} \subset \Tilde{B}\) and \(\inf_{\Tilde{B}} \abs{u} \ge \frac{1}{2} \inf_B \abs{u}\),
Take \(\eta \in C^\infty_c (\R^N)\) such that \(\eta = 1\) on \(B_1\). Write
\(
  I_\alpha = \eta I_\alpha + (1 - \eta) I_\alpha.
\)
Note that since \(u \in L^s (\R^N)\) for every \(s \ge 1\)
\[
 ((1 - \eta) I_\alpha ) \ast \abs{u}^p \in C^\infty (\R^N).
\]
On the other hand, if either \(k + \lambda < p - 1\) or \(\inf_B \abs{u} > 0\), \(\abs{u}^p, \abs{u}^{p - 2}u \in C^{k, \lambda}(\Tilde{B})\).
Since \(\eta I_\alpha \in L^1 (\R^N)\), \((\eta I_\alpha) \ast \abs{u}^p \in C^{k, \lambda} (\Tilde{B})\).
Therefore \((I_\alpha \ast \abs{u}^p) \abs{u}^{p - 2} u \in C^{k, \lambda} (\Tilde{B})\) and by the classical Schauder regularity estimates \cite{Gilbarg}*{theorem 4.6}, \(u \in C^{2, \lambda} (B)\).
\end{proof}

\section{Positivity and symmetry}

\subsection{Positivity}
\label{sectionPositivity}

We show that any groundstate of \eqref{eqChoquard} is a positive function.

\begin{proposition}
\label{propositionPositivity}
Let \(\alpha \in (0, N)\) and \(p \in (1, \infty)\) be such that \(\frac{N - 2}{N + \alpha} < \frac{1}{p} < \frac{N}{N + \alpha}\).
If \(u \in W^{1, 2}(\R^N)\) is a groundstate of
\[
 -\Delta u + u = (I_\alpha \ast \abs{u}^p) \abs{u}^{p - 2} u,
\]
then either
\(
 u > 0
\)
or
\(
 u < 0.
\)
\end{proposition}
\begin{proof}
If \(u\) is a groundstate note that \(\abs{u}\) is also a groundstate, and therefore by proposition~\ref{propositionRegularity} one has \(\abs{u} \in C^2 (\R^N)\) and
\[
  - \Delta \abs{u} + \abs{u} = (I_\alpha \ast \abs{u}^p)\abs{u}^{p - 1}.
\]
By the strong maximum principle, either \(\abs{u} > 0\) or \(\abs{u} = 0\). Since \(u \ne 0\), we conclude that either \(u > 0\) or \(u < 0\).
\end{proof}

Proposition~\ref{propositionPositivity} can also be proved without using any regularity of the groundstates. Indeed, if \(u \in H^1 (\R^N)\) and \(- \Delta \abs{u} + \abs{u} \ge 0\) weakly, then either \(u > 0\), \(u < 0\) or \(u = 0\) almost everywhere \cite{VanSchaftingenWillem2008}*{proposition 2.1}.

\subsection{Symmetry}
\label{sectionSymmetry}

In this section we prove that any groundstate of \eqref{eqChoquard} is a radial and radially decreasing solution.

\begin{proposition}
\label{propositionSymmetry}
Let \(\alpha \in (0, N)\) and \(p \in (1, \infty)\) be such that \(\frac{N - 2}{N + \alpha} < \frac{1}{p} < \frac{N}{N + \alpha}\).
If \(u \in W^{1, 2}(\R^N)\) is a positive groundstate of
\[
 -\Delta u + u = (I_\alpha \ast \abs{u}^p) \abs{u}^{p - 2} u,
\]
then there exist \(x_0 \in \R^n\) and \(v : (0, \infty) \to \R\) a nonnegative nonincreasing function such that for almost every \(x \in \R^N\)
\[
 u (x) = v (\abs{x - x_0}).
\]
\end{proposition}

By proposition \ref{propositionPositivity}, any groundstate of \eqref{eqChoquard} is positive.
Therefore, when \(p \ge 2\) and assumptions \eqref{conditionMaZhao}
holds, proposition~\ref{propositionSymmetry} follows from the result of Li Ma and Lin Zhao \cite{Ma-Zhao-2010}*{theorem 2}.

Our proof will follow the strategy of T.\thinspace Bartsch, M.\thinspace Willem and T.\thinspace Weth \citelist{\cite{BartschWethWillem2005}\cite{VanSchaftingenWillem2008}}, which consists in using the minimality property of the groundstate to deduce some relationship between the function and its polarization. The argument is simplified here because the nonlocal term has some strong symmetrizing effect.

Let \(H \subset \R^N\) be a closed half-space. Let \(\sigma_H\) denote the reflection with respect to \(\partial H\) and \(u^H : \R^N\to \R\) be the polarization of \(u : \R^N \to \R\) defined for \(x \in \R^N\) by
\[
  u^H (x)
= \begin{cases}
    \max \bigl(u(x), u (\sigma_H (x)\bigr) & \text{if \(x \in H\)},\\
    \min \bigl(u(x), u (\sigma_H (x)\bigr) & \text{if \(x \not \in H\)},
  \end{cases}
\]
(see \cite{BrockSolynin2000}).

A first tool is the study of equality cases in a polarization inequality.

\begin{lemma}
\label{lemmaEqualityMaster}
Let \(\alpha \in (0, N)\) and \(u \in L^{\frac{2 N}{N + \alpha}} (\R^N)\) and  \(H \subset \R^N\) be a closed half-space.
If \(u \ge 0\) and
\[
\int_{\R^N} \int_{\R^N} \frac{u(x)\, u(y)}{\abs{x - y}^{N - \alpha}} \,dx\,dy\\
\ge \int_{\R^N} \int_{\R^N} \frac{ u^H(x)\, u^H(y)}{\abs{x - y}^{N - \alpha}} \,dx\,dy,
\]
then either \(u^H = u\) or \(u^H = u \circ \sigma_H\).
\end{lemma}

This inequality is proved without the equality cases by A.\thinspace Baernstein II \cite{Baernstein1994}*{corollary 4} (see also \cite{VanSchaftingenWillem2004}*{proposition 8}.

\begin{proof}
First note that the integrals are finite by the Hardy--Littlewood--Sobolev inequality. Also note that
\begin{multline*}
  \int_{\R^N} \int_{\R^N} \frac{u(x)\, u(y)}{\abs{x - y}^{N - \alpha}} \,dx\,dy\\
  = \int_{H} \int_{H} \frac{u(x) u (y) + u \bigl(\sigma_H(x)\bigr) u \bigl(\sigma_H(y)\bigr)}{\abs{x - y}^{N - \alpha}} \\
+  \frac{u (x) u \bigl(\sigma_H(y)\bigr) + u \bigl(\sigma_H(x)\bigr) u (y)}{\abs{x - \sigma_H (y)}^{N - \alpha}} \,dx\,dy
\end{multline*}
and that a similar identity holds for \(u^H\).
By an inspection of all the possible cases, we have, since \(u \ge 0\) and \(\alpha < N\), for every \(x \in H\) and \(y \in H\)
\begin{multline*}
 \frac{u(x) u (y) + u \bigl(\sigma_H(x)\bigr) u \bigl(\sigma_H(y)\bigr)}{\abs{x - y}^{N - \alpha}}
+  \frac{u (x) u \bigl(\sigma_H(y)\bigr) + u \bigl(\sigma_H(x)\bigr) u (y)}{\abs{x - \sigma_H (y)}^{N - \alpha}} \\
\le
\frac{u^H(x) u^H (y) + u^H\bigl(\sigma_H (x)\bigr) u^H \bigl(\sigma_H(y)\bigr)}{\abs{x - y}^{N - \alpha}}\\
+  \frac{u^H (x) u^H \bigl(\sigma_H(y)\bigr) + u^H \bigl(\sigma_H (x)\bigr) u^H (y)}{\abs{x - \sigma_H (y)}^{N - \alpha}},
\end{multline*}
with equality if and only if either \(u^H (x) = u (x)\) and \(u^H (y) = u (y)\) or \(u^H (x) = u \bigl(\sigma_H (x)\bigr)\) and \(u^H (y) = u \bigl(\sigma_H (y)\bigr)\).
Since this holds for almost every \( (x, y) \in H^2\), we have proved the lemma.
\end{proof}

Now we show how the information that either \(u^H = u\) or \(u^H = u \circ \sigma_H\) can be used
to deduce some symmetry.
\begin{lemma}
\label{lemmaSymmetry}
Let \(s \ge 1\) and \(u \in L^s (\R^N)\). If \(u \ge 0\) and for every closed half-space \(H \subset \R^N\), \(u^H = u\) or \(u^H = u \circ \sigma_H\), then there exist \(x_0 \in \R^N\) and \(v : (0, \infty) \to \R\) a nondecreasing function such that for almost every \(x \in \R^N\), \(u (x) = v (\abs{x - x_0})\).
\end{lemma}

This was already proved \cite{VanSchaftingenWillem2008}*{proposition 3.15}. We give here a simpler argument that does not rely on the approximation of symmetrizations by polarizations.
In order to prove lemma~\ref{lemmaSymmetry}, we recall the following result.
\begin{lemma}
\label{lemmaPolarizationHardyLittlewoodStrict}
Let \(u \in L^s(\R^N)\), \(w \in L^t (\R^N)\) with \(\frac{1}{s} + \frac{1}{t} = 1\) to be a radial function such that for every \(x, y \in \R^N\) with \(\abs{x} \le \abs {y}\), \(w (x) \ge w (y)\) with equality if and only if \(\abs{x} = \abs{y}\), and let \(H \subset \R^N\) be a closed half-space. If \(0\) is an interior point of \(H\) and
\[
 \int_{\R^N} u^H w \le \int_{\R^N} u w,
\]
then \(u = u^H\).
\end{lemma}

The proof is elementary and was given for example in \cite{VanSchaftingen2009}*{lemma 4}.
We also use a classical characterization of functions invariant under polarizations \citelist{\cite{BrockSolynin2000}*{lemma 6.3}\cite{VanSchaftingen2009}*{lemma 5}}.

\begin{lemma}
\label{lemmaPolarizationSymmetrization}
Let \(u \in L^s(\R^N)\). If for every closed half-space \(H \subset \R^N\) such that \(x_0 \in H\), \(u^H = u\), then there exists \(v : (0, \infty) \to \R\) a nondecreasing function such that for almost every \(x \in \R^N\), \(u (x) = v (\abs{x - x_0})\).
\end{lemma}

\begin{proof}[Proof of lemma~\ref{lemmaSymmetry}]
Choose \(w\) that satisfies the assumptions of lemma~\ref{lemmaPolarizationHardyLittlewoodStrict} and define for \(x \in \R^N\) the function
\[
 W (x) = \int_{\R^N} u (y) w (x - y) \,dy.
\]
This function is nonnegative, continuous and \(\lim_{\abs{x} \to \infty} W (x) = 0\), and therefore admits a maximum point \(x_0 \in \R^N\).

In order to conclude with lemma~\ref{lemmaPolarizationSymmetrization}, we now claim that for every closed half-space \(H\) such that \(x_0 \in H\), \(u^H = u\).
Indeed, if \(u^H = u \circ \sigma_H\) one has since \(w\) is radial and by definition of \(x_0\)
\[
\begin{split}
  \int_{\R^N} u^H (y) w (x_0 - y) \,dy &= \int_{\R^N} u \bigl(\sigma_H (y)\bigr) w (x_0 - y) \,dy\\
& = \int_{\R^N} u (y) w \bigl(\sigma_H(x_0) - y\bigr) \,dy\\
& \le \int_{\R^N} u (y) w (x_0 - y) \,dy
\end{split}
\]
Therefore, by lemma~\ref{lemmaPolarizationHardyLittlewoodStrict}, if \(x_0\) is an interior point of \(H\), we have \(u^H = u\). The case \(x_0 \in \partial H\) follows by continuity.
\end{proof}

\begin{proof}[Proof of proposition~\ref{propositionSymmetry}]
Let \(H \subset \R^N\) be a closed half-space.
Note that
\[
 \int_{\R^N} \abs{\nabla u^H}^2 = \int_{\R^N} \abs{\nabla u}^2
\]
and
\[
 \int_{\R^N} \abs{u^H}^2 = \int_{\R^N} \abs{u}^2.
\]
In view of the characterization of grounds states, we have necessarily
\[
 \int_{\R^N} (I_\alpha \ast \abs{u^H}^p) \abs{u^H}^p \ge
 \int_{\R^N} (I_\alpha \ast \abs{u}^p) \abs{u}^p.
\]
By lemma~\ref{lemmaEqualityMaster}, for every closed half-space \(H \subset \R^N\), either \(u^H = u\) or \(u^H = u \circ \sigma_H\). We conclude by lemma~\ref{lemmaSymmetry}.
\end{proof}

\section{Decay asymptotics}
\label{sectionAsymptotics}

In this last section, we study the asymptotic behavior of solutions.

\subsection{Asymptotics of the nonlocal term}
We first study the asymptotics of the nonlocal term.
We give explicit bounds on convergence rates that we shall need in the sequel.

\begin{proposition}
\label{propositionRieszAsymptotics}
Let \(N\in \N_*\), \(\alpha \in (0, N)\) and \(p \in (1, \infty)\).
If \(\frac{N - 2}{N+\alpha} < \frac{1}{p} < \frac{N}{N + \alpha}\) and \(u\) is a groundstate of
\[
  -\Delta u + u =(I_\alpha \ast \abs{u}^p)\abs{u}^{p - 2} u\quad\text{in \(\R^N\)},
\]
then there exists \(C \in \R\) such that for every \(x \in \R^N\),
\[
\Bigabs{I_\alpha \ast \abs{u}^p (x) - I_\alpha (x) \int_{\R^N} \abs{u}^p}
\le \frac{C}{\abs{x}^{N - \alpha}} \Bigl( \frac{1}{1 + \abs{x}} + \frac{1}{1 + \abs{x}^{N (p - 1)}}\Bigr).
\]
\end{proposition}

The proof of this proposition will follow from regularity estimates together with a computation of the asymptotics of a Riesz potential.

\begin{lemma}
\label{lemmaRieszAsymptotics}
Let \(\alpha \in (0, N)\), \(\beta \in (N, \infty)\) and \(f \in L^\infty (\R^N)\).
If
\[
  \sup_{x \in \R^N} \abs{f(x)}\, \abs{x}^\beta < \infty,
\]
then there exists \(C > 0\) such that for every \(x \in \R^N\)
\[
  \Bigabs{\int_{\R^N} \frac{f (y)}{\abs{x - y}^{N - \alpha}} \,dy
- \frac{1}{\abs{x}^{N - \alpha}} \int_{\R^N} f (y) \,dy}
\le \frac{C}{\abs{x}^{N - \alpha}} \Bigl( \frac{1}{1 + \abs{x}} + \frac{1}{1 + \abs{x}^{\beta - N}}\Bigr).
\]
\end{lemma}

Similar statements have already appeared in \cite{GidasNiNirenberg1981}*{lemma 2.1}.
Here we emphasize the precise control on the rate of decay at infinity.
The reader will see in the proof that \(C\) only depends on \(\sup_{x \in \R^N} \abs{f(x)} (1 + \abs{x}^\beta)\).

\begin{proof}[Proof of lemma~\ref{lemmaRieszAsymptotics}]
We need to estimate the quantity
\[
 \Bigabs{\int_{\R^N} f (y) \Bigl(\frac{1}{\abs{x - y}^{N - \alpha}} - \frac{1}{\abs{x}^{N - \alpha}} \Bigr) \,dy}
\le  \int_{\R^N} \abs{f (y)} \Bigabs{\frac{1}{\abs{x - y}^{N - \alpha}} - \frac{1}{\abs{x}^{N - \alpha} }} \,dy.
\]
On the one hand, there exists \(C \in \R\) such that if \(x, y \in \R^N\) and \(\abs{y} \le 2 \abs{x}\), one has
\[
 \Bigabs{\frac{1}{\abs{x - y}^{N - \alpha}} - \frac{1}{\abs{x}^{N - \alpha}}}
  \le \frac{C \abs{y}}{\abs{x}^{N - \alpha + 1}}
\]
and thus
\begin{equation}
\label{eqRieszAsymptoticsSmall}
\begin{split}
\Bigabs{\int_{B_{2 \abs{x}}} f (y) \Bigl(\frac{1}{\abs{x - y}^{N - \alpha}} - \frac{1}{\abs{x}^{N - \alpha}} \Bigr) \,dy}
& \le
 \Bigabs{\frac{1}{\abs{x}^{N - \alpha + 1}} \int_{B_{2 \abs{x}}} \frac{C' \abs{y}}{1 + \abs{y}^\beta}\,dy}\\
&\le \frac{C''}{\abs{x}^{N - \alpha}} \Bigl( \frac{1}{1 + \abs{x}} + \frac{1}{1 + \abs{x}^{\beta - N}}\Bigr).
\end{split}
\end{equation}
On the other hand, there exists \(C \in \R\) such that if \(x, y \in \R^N\) and \(\abs{y} \ge 2 \abs{x}\),
\[
 \Bigabs{\frac{1}{\abs{x - y}^{N - \alpha}} - \frac{1}{\abs{x}^{N - \alpha}}}
  \le \frac{C}{\abs{x}^{N - \alpha}},
\]
from which we compute that
\begin{equation}
\label{eqRieszAsymptoticsLarge}
\begin{split}
\Bigabs{\int_{\R^N \setminus B_{2 \abs{x}}} f (y) \Bigl(\frac{1}{\abs{x - y}^{N - \alpha}} - \frac{1}{\abs{x}^{N - \alpha}} \Bigr) \,dy}
& \le
 \Bigabs{\frac{1}{\abs{x}^{N - \alpha}}\int_{\R^N \setminus B_{2 \abs{x}}} \frac{C'}{\abs{y}^\beta}  \,dy}\\
& \le \frac{C''}{\abs{x}^{N - \beta}}.
\end{split}
\end{equation}
The conclusion follows from \eqref{eqRieszAsymptoticsSmall} and \eqref{eqRieszAsymptoticsLarge}.
\end{proof}

\begin{proof}[Proof of proposition~\ref{propositionRieszAsymptotics}]
By proposition~\ref{propositionRegularity}, \(u \in L^1 (\R^N) \cap L^\infty (\R^N)\) and by proposition~\ref{propositionSymmetry}, \(\abs{u}\) is  radial and radially decreasing. Therefore for every \(x \in \R^N\),
\[
 \abs{u (x)} \le \frac{1}{\abs{B_{\abs{x}}}} \int_{\R^N} \abs{u}
\le \frac{C}{\abs{x}^{N}}.
\]
The conclusion follows from the application of lemma~\ref{lemmaRieszAsymptotics} with \(f = \abs{u}^p\) and \(\beta = N p\).
\end{proof}

\subsection{Superlinear case}
In the case \(p > 2\), we are going to show that groundstates have classical exponential decay.

\begin{proposition}
\label{propositionSuperlinearAsymptotics}
Let \(N \in \N\),  \(\alpha \in (0, N)\) and \(p \in (2, \infty)\). If \(\frac{1}{p} < \frac{N - 2}{N + \alpha}\)
and \(u\) is a positive groundstate of
\[
  -\Delta u + u =(I_\alpha \ast \abs{u}^p)\abs{u}^{p - 2} u\quad\text{in \(\R^N\)},
\]
then
\[
  \lim_{\abs{x} \to \infty} u (x) \abs{x}^{\frac{N - 1}{2}} e^{\abs{x}} \in (0, \infty).
\]
\end{proposition}

Note that if \(\alpha \le N - 4\) then the assumptions of the proposition cannot be satisfied.
S.\thinspace Cingolani, M.\thinspace Clapp and S.\thinspace Secchi have proved that the limit is finite \cite{CingolaniClappSecchi2011}*{lemma A.2}.

The proof of this result follows the standard proof for the nonlocal problem.
Our main tool computational tool is an analysis of the decay rate of solutions of a linear Schr\"odinger equation.

\begin{lemma}
\label{lemmaSolutionLinear}
Let \(\rho \ge 0\) and \(W \in C^1( (\rho, \infty), \R)\).
If
\[
 \lim_{s \to \infty} W (s) > 0.
\]
and for some \(\beta > 0\),
\[
 \lim_{s \to \infty} W'(s)s^{1 + \beta} = 0,
\]
then there exists a nonnegative radial function \(v : \R^N \setminus B_\rho \to \R\)
such that
\[
 -\Delta v + W v = 0
\]
in \(\R^N \setminus B_\rho\) and for some \(\rho_0 \in (\rho, \infty)\),
\[
\lim_{\abs{x}\to \infty} v(x)\abs{x}^\frac{N - 1}{2}
\exp \int_{\rho_0}^{\abs{x}} \sqrt{W} =1.
\]
\end{lemma}

S.\thinspace Agmon has proved this result when \(\beta > \frac{1}{2}\) \cite{Agmon}*{theorem 3.3}.
The proof of lemma~\ref{lemmaSolutionLinear} appears in a previous work by the authors \cite{MorozVanSchaftingen}*{proposition 6.1}. The proof is based on comparison with functions of the form
\[
 v_\tau (x) = \abs{x}^{-\frac{N - 1}{2}}
\exp \Bigl(- \int_{\rho}^{\abs{x}} \sqrt{W} + \tau \abs{x}^{\beta} \Bigr)
\]
for suitable values of \(\tau \in \R\).

\begin{proof}[Proof of proposition~\ref{propositionSuperlinearAsymptotics}]
By proposition~\ref{propositionRegularity}, \(u \in L^1 (\R^N) \cap L^\infty (\R^N)\).
Since \(p > 2\), we have
\[
 \lim_{\abs{x} \to \infty} \bigl(I_\alpha \ast \abs{u}^p\bigr) (x) \abs{u (x)}^{p - 2} = 0.
\]
Hence, there exists \(\rho \in \R\) such that in \(x \in \R^N\) and \(\abs{x} \ge \rho\),
\[
 \bigl(I_\alpha \ast \abs{u}^p\bigr) (x) \abs{u (x)}^{p - 2} \le \frac{3}{4}.
\]
We have thus in \(\R^N \setminus B_\rho\),
\[
 - \Delta u + \frac{1}{4} u \le 0.
\]
Let \(v \in C^2 (\R^N \setminus B_\rho, \R)\) be such that
\[
\left\{
\begin{aligned}
  - \Delta v + \frac{1}{4} v & = 0 & & \text{if \(x \in \R^N \setminus B_\rho\)},\\
  v (x) & = u (x) & & \text{if \(x \in \partial B_\rho\)},\\
  \lim_{\abs{x} \to \infty} v (x) & = 0.
\end{aligned}
\right.
\]
By lemma~\ref{lemmaSolutionLinear} with \(W = \frac{1}{4}\), there exists \(\mu \in \R\) such that for every \(x \in \R^N \setminus B_\rho\)
\[
 v (x) \le \frac{\mu}{\abs{x}^{\frac{N - 1}{2}}} e^{-\frac{\abs{x}}{2}}.
\]
Hence, by the comparison principle, for every \(x \in \R^N \setminus B_\rho\)
\[
 u (x) \le v(x) \le \frac{\mu}{\abs{x}^{\frac{N - 1}{2}}} e^{-\frac{\abs{x}}{2}},
\]
which implies that there exists \(\nu \in \R\) such that for every \(x \in \R^N \setminus B_\rho\)
\[
 \bigl(I_\alpha \ast \abs{u}^p\bigr) (x) \abs{u (x)}^{p - 2} \le \nu e^{-\frac{p - 2}{2} \abs{x}}.
\]

We have now
\[
 -\Delta u + u \ge 0 \ge -\Delta u + W u,
\]
in \(\R^N \setminus B_\rho\),
where \(W \in C^1 (\R^N \setminus B_\rho)\) is defined for \(x \in \R^N \setminus \rho\) by
\[
  W (x) =  1 - \nu e^{-\frac{p - 2}{2} \abs{x}}.
\]
Define now \(\underline{u}, \overline{u} \in C^2 (\R^N \setminus B_\rho, \R)\) such that
\[
\left\{
\begin{aligned}
  - \Delta \underline{u} + \underline{u} & = 0 & & \text{if \(x \in \R^N \setminus B_\rho\)},\\
  \underline{u} (x) & = u (x) & & \text{if \(x \in \partial B_\rho\)},\\
  \lim_{\abs{x} \to \infty} \underline{u} (x) & = 0,
\end{aligned}
\right.
\]
and
\[
\left\{
\begin{aligned}
  - \Delta \overline{u} + W \overline{u} & = 0 & & \text{if \(x \in \R^N \setminus B_\rho\)},\\
  \overline{u} (x) & = u (x) & & \text{if \(x \in \partial B_\rho\)},\\
  \lim_{\abs{x} \to \infty} \overline{u} (x) & = 0.
\end{aligned}
\right.
\]
By the comparison principle
\(\underline u \le u \le \overline u\) in \(\R^N \setminus B_\rho\),
whence in view of  lemma~\ref{lemmaSolutionLinear}
\begin{equation}
\label{eqSuperlinearAsymptoticsLiminfLimsup}
\begin{split}
 0 < \lim_{\abs{x} \to \infty}  \underline{u}(x)\abs{x}^\frac{N - 1}{2}
e^{\abs{x}}
& \le \liminf_{\abs{x} \to \infty}  u(x)\abs{x}^\frac{N - 1}{2}
e^{\abs{x}} \\
& \le \limsup_{\abs{x} \to \infty}  u(x)\abs{x}^\frac{N - 1}{2}
e^{\abs{x}}
\le
 \lim_{\abs{x} \to \infty}  \overline{u}(x)\abs{x}^\frac{N - 1}{2}
e^{\abs{x}} < \infty.
\end{split}
\end{equation}
It remains to prove that the \(u(x)\abs{x}^\frac{N - 1}{2}
e^{\abs{x}}\) has a limit as \(x \to \infty\). Since \(u\) and \(\underline{u}\), are both radial functions, we have by the comparison principle, for every \(r, s \in (\rho, \infty)\) such that \(r \le s\),
\[
 \frac{u(r)}{\underline{u}(r)} \le \frac{u(s)}{\underline{u}(s)},
\]
that is, the function \(u / \underline{u}\) is nondecreasing.
On the other hand, by \eqref{eqSuperlinearAsymptoticsLiminfLimsup}, the function \(u / \underline{u}\) is bounded. We conclude then that \(u/\underline{u}\) has a finite limit at infinity and that
\[
  \lim_{\abs{x} \to \infty}  u (x)\abs{x}^\frac{N - 1}{2}
e^{\abs{x}}
  = \lim_{\abs{x} \to \infty} \frac{u (x)}{\underline{u} (x)} \lim_{\abs{x} \to \infty}  \underline{u}(x)\abs{x}^\frac{N - 1}{2}
e^{\abs{x}} \in (0, \infty).\qedhere
\]
\end{proof}

\subsection{Linear case}
The analysis of the case \(p > 2\) is based on  the fact that \(\abs{u}^{p - 2}\) decays exponentially at infinity.
When \(p = 2\), we have to take into account the Riesz potential \(I_\alpha \ast \abs{u}^2\).

\begin{proposition}
\label{propositionLinearAsymptotics}
Let \(N \in \N\), \(\alpha \in (0, N)\) and \(\alpha>N-4\).
If \(u\) is a positive groundstate of
\[
  -\Delta u + u =(I_\alpha \ast \abs{u}^2)u\quad\text{in \(\R^N\)},
\]
then
\[
  \lim_{\abs{x} \to \infty} u (x) \abs{x}^{\frac{N - 1}{2}} \exp \int_\nu^{\abs{x}} \sqrt{1 - \frac{\nu^{N - \alpha}}{s^{N - \alpha}}} \,ds \in (0, \infty),
\]
where
\[
  \nu^{N - \alpha} = \frac{\Gamma(\tfrac{N-\alpha}{2})}{\Gamma(\tfrac{\alpha}{2})\pi^{N/2}2^{\alpha}} \int_{\R^N} \abs{u}^p.
\]
\end{proposition}

\begin{remark}\label{remarkLinearAsymptotics}
The asymptotics of the groundstate are thus related to the behavior of the function
\[
  \int_\rho^{\abs{x}} \sqrt{1 - \frac{\nu^{N - \alpha}}{s^{N - \alpha}}}\,ds
\]
as \(\abs{x} \to \infty\).
When \(\alpha < N - 1\), we have
\[
  \int_\rho^{\abs{x}} 1 - \sqrt{1 - \frac{\nu^{N - \alpha}}{s^{N - \alpha}}}\,ds < \infty,
\]
and therefore,
\[
 \lim_{\abs{x} \to \infty} u (x) \abs{x}^{\frac{N - 1}{2}} e^{\abs{x}} \in (0, \infty),
\]
as in proposition~\ref{propositionSuperlinearAsymptotics}.
When \(\alpha < N - \frac{1}{2}\), we have
\[
 \int_\rho^{\abs{x}} 1 - \frac{\nu^{N - \alpha}}{2 s^{N - \alpha}} - \sqrt{1 - \frac{\nu^{N - \alpha}}{s^{N - \alpha}}}\,ds < \infty.
\]
Therefore, in the critical case \(\alpha = N - 1\) which includes the physical case \(N = 3\) and \(\alpha = 2\), we have a polynomial perturbation of the previous asymptotics,
\[
 \lim_{\abs{x} \to \infty} u (x) \abs{x}^{\frac{N - 1 - \nu}{2}} e^{\abs{x}} \in (0, \infty),
\]
whereas when \(\alpha \in (N - 1,N - \frac{1}{2})\),
\[
 \lim_{\abs{x} \to \infty} u (x) \abs{x}^{\frac{N - 1}{2}} e^{\abs{x}} e^{-\frac{\nu^{N - \alpha} \abs{x}^{1 - (N - \alpha)}} {2 (1 - (N - \alpha))}} \in (0, \infty).
\]
Larger values of \(\alpha\in[N - \frac{1}{2},N)\) could be analyzed by taking higher-order Taylor expansions of the square root \cite{MorozVanSchaftingen}*{remark 6.1}.
\end{remark}

We are now ready to establish the asymptotics of the solution in the linear case \(p = 2\).

\begin{proof}[Proof of proposition~\ref{propositionLinearAsymptotics}]
By proposition~\ref{propositionRieszAsymptotics}, there exists \(\mu > 0\) such that for every \(x \in \R^N\),
\[
   - \frac{\mu}{\abs{x}^{N - \alpha + 1}} \le \bigl(I_\alpha \ast \abs{u}^2\bigr) (x) - \frac{\nu^{N - \alpha}}{\abs{x}^{N - \alpha}}\le \frac{\mu}{\abs{x}^{N - \alpha + 1}}.
\]
Define now \(\underline{W} \in C^1 (\R^N \setminus \{0\})\) and \(\overline{W} \in C^1 (\R^N \setminus \{0\})\) for \(x \in \R^N \setminus \{0\}\) by
\[
  \underline{W} (x) = 1 - \frac{\nu^{N - \alpha}}{\abs{x}^{N - \alpha}} + \frac{\mu}{\abs{x}^{N - \alpha + 1}}
\]
and
\[
  \overline{W} (x) = 1 - \frac{\nu^{N - \alpha}}{\abs{x}^{N - \alpha}} - \frac{\mu}{\abs{x}^{N - \alpha + 1}}.
\]
One has, in \(\R^N \setminus \{0\}\),
\[
  - \Delta u + \overline{W} u \le 0 \le - \Delta u + \underline{W} u.
\]
Define now \(\underline{u}, \overline{u} \in C^2 (\R^N \setminus B_\rho, \R)\) such that
\[
\left\{
\begin{aligned}
  - \Delta \underline{u} + \underline{W}\,\underline{u} & = 0 & & \text{if \(x \in \R^N \setminus B_1\)},\\
  \underline{u} (x) & = u (x) & & \text{if \(x \in \partial B_1\)},\\
  \lim_{\abs{x} \to \infty} \underline{u} (x) & = 0,
\end{aligned}
\right.
\]
and
\[
\left\{
\begin{aligned}
  - \Delta \overline{u} + \overline{W} \overline{u} & = 0 & & \text{if \(x \in \R^N \setminus B_1\)},\\
  \overline{u} (x) & = u (x) & & \text{if \(x \in \partial B_1\)},\\
  \lim_{\abs{x} \to \infty} \overline{u} (x) & = 0.
\end{aligned}
\right.
\]
One has for \(\rho\) large enough,
\[
 \lim_{\abs{x} \to \infty}
    \underline{u} (x) \abs{x}^\frac{N - 1}{2} \int_{\rho}^{\abs{x}} \sqrt{\underline{W}(s)} \,ds \in (0, \infty)
\]
and
\[
  \lim_{\abs{x} \to \infty}
    \overline{u} (x) \abs{x}^\frac{N - 1}{2} \int_{\rho}^{\abs{x}} \sqrt{\overline{W}(s)}\,ds \in (0, \infty).
\]
Note now that
\[
 \int_{\rho}^\infty \sqrt{\underline{W}} - \sqrt{\overline{W}} < \infty,
\]
from which we deduce that
\[
 \lim_{\abs{x} \to \infty}
    \overline{u} (x) \abs{x}^\frac{N - 1}{2} \int_{\rho}^{\abs{x}} \sqrt{1 - \frac{\nu^{N - \alpha}}{s^{N - \alpha}}} \,ds \in (0, \infty)
\]
and
\[
 \lim_{\abs{x} \to \infty}
    \underline{u} (x) \abs{x}^\frac{N - 1}{2} \int_{\rho}^{\abs{x}} \sqrt{1 - \frac{\nu^{N - \alpha}}{s^{N - \alpha}}} \,ds \in (0, \infty).
\]
We conclude that
\[
 \lim_{\abs{x} \to \infty}
    u (x) \abs{x}^\frac{N - 1}{2} \int_{\rho}^{\abs{x}} \sqrt{1 - \frac{\nu^{N - \alpha}}{s^{N - \alpha}}} \,ds \in (0, \infty)
\]
as in the proof of proposition~\ref{propositionSuperlinearAsymptotics}.
\end{proof}

\subsection{Sublinear case}

Whereas in the case \(p > 2\), the nonlinear term did not play any role in the asymptotics and for \(p = 2\) all the terms were playing a role, in the case \(p < 2\), the asymptotics are governed by the terms without derivatives in the equation.

\begin{proposition}
\label{propositionSublinearAsymptotics}
Let \(N \in \N\),  \(\alpha \in (0, N)\) and \(p \in (1,2)\).
If \(\frac{N}{N + \alpha}<\frac{1}{p} < \frac{N - 2}{N + \alpha}\) and
and \(u\) is a positive groundstate of
\[
  -\Delta u + u =(I_\alpha \ast \abs{u}^p)\abs{u}^{p - 2} u\quad\text{in \(\R^N\)},
\]
then
\[
  \lim_{\abs{x} \to \infty} \bigl(u (x)\bigr)^{2 - p} \abs{x}^{N - \alpha} = \frac{\Gamma(\tfrac{N-\alpha}{2})}{\Gamma(\tfrac{\alpha}{2})\pi^{N/2}2^{\alpha}} \int_{\R^N} \abs{u}^p.
\]
\end{proposition}

When \(u\) is merely a distributional supersolution to the equation, it was already known that
\cite{MorozVanSchaftingen}*{theorem 5}
\[
  \liminf_{\abs{x} \to \infty} \frac{\bigl(u (x)\bigr)^{2 - p}}{\abs{x}^{N - \alpha}} > 0.
\]

In the limiting case \(p = \frac{N + \alpha}{N}\), Chen Weng, Li Congming and Ou Biao \cite{ChenLiOu2006} have proved that if \(u \in L^\frac{2 N}{N - \alpha} (\R^N)\) is nonnegative and satisfies the equation without the differential operator
\[
 u = (I_\alpha \ast \abs{u}^\frac{N + \alpha}{N}) u^\frac{\alpha}{N},
\]
then
\[
 u(x)^\frac{N - \alpha}{N} = \int_{\R^N} \abs{u}^\frac{N + \alpha}{N}\frac{C_\alpha}{(\lambda^2 + \abs{x}^2)^\frac{N - \alpha}{2}}.
\]

In the proof of proposition~\ref{propositionSublinearAsymptotics}, we shall use the asymptotics of the nonlocal term
derived in proposition~\ref{propositionRieszAsymptotics}, together with asymptotics of solutions to some linear equations.

\begin{lemma}
\label{lemmaLinearPowerAsymptotics}
Let \(\rho > 0\), \(\beta > 0\) and \(\lambda > 0\) and \(u \in C^2 (\R^N \setminus B_\rho)\).
If for every \(x \in \R^N \setminus B_\rho\),
\[
 - \Delta u (x) + \lambda u(x) = \frac{1}{\abs{x}^\beta},
\]
and
\[
 \lim_{\abs{x} \to \infty} u (x) = 0,
\]
then
\[
 \lim_{\abs{x} \to \infty}  u (x) \lambda \abs{x}^\beta = 1.
\]
\end{lemma}
\begin{proof}
Let \(w \in C^2 (\R^N \setminus B_\rho)\) be a solution of \(-\Delta w + w = 0\) such that \(w > 0\) and \(\lim_{\abs{x} \to \infty} w (x) = 0\).
For \(\sigma, \tau \in \R\) define \(v_{\sigma, \tau} \in C^2 (\R^N \setminus B_\rho)\) for every \(x \in \R^N \setminus B_\rho\) by
\[
 v_{\sigma, \tau} (x) = \frac{1}{\lambda \abs{x}^\beta} + \frac{\sigma}{\abs{x}^{\beta + 2}} + \tau w (x) .
\]
One has, for every \(\sigma, \tau \in \R\) and \(x \in \R^N \setminus B_\rho\),
\[
 - \Delta v_{\sigma, \tau} (x) + v_{\sigma, \tau} (x)
= \frac{1}{\abs{x}^\beta}
    + \frac{1}{\abs{x}^{\beta + 2}}
          \biggl(\frac{\beta (N - 2 - \beta)}{\lambda}
                 + \sigma \Bigl( \lambda + \frac{(\beta + 2) (N - \beta)}{\abs{x}^2} \Bigr)
          \biggr).
\]
Let \(R \in (0, \infty)\) such that
\[
 \frac{\abs{\beta + 2} \abs{N - \beta}}{R^2} \le \frac{\lambda}{2}.
\]
If we choose now \(\underline{\sigma}, \overline{\sigma} \in \R\) such that
\[
  \underline{\sigma} \frac{3\lambda}{2} \le \frac{\beta (\beta + 2 - N)}{\lambda}
\le \overline{\sigma} \frac{\lambda}{2},
\]
for every \(x \in \R^N \setminus B_R\) and \(\tau \in \R\), we have
\[
 -\Delta v_{\underline{\sigma}, \tau} (x) + \lambda v_{\underline{\sigma}, \tau} (x)
  \le \frac{1}{\abs{x}^\beta}
  \le -\Delta v_{\overline{\sigma}, \tau} (x) + \lambda v_{\overline{\sigma}, \tau} (x)
\]
Since \(u\) is continuous, \(u\) is bounded on \(\partial B_R\) and there exist \(\underline{\tau}, \overline{\tau} \in \R\) such that
\(
 v_{\underline{\sigma}, \underline{\tau}} \le u \le v_{\overline{\sigma}, \overline{\tau}}
\)
on \(\partial B_R\). By the comparison principle,
\(
  v_{\underline{\sigma}, \underline{\tau}} \le u \le v_{\overline{\sigma}, \overline{\tau}}
\)
in \(\R^N \setminus B_R\). Since \(\lim_{\abs{x} \to \infty} v_{\underline{\sigma}, \underline{\tau}} (x) \lambda \abs{x}^\beta = \lim_{\abs{x} \to \infty} v_{\overline{\sigma}, \overline{\tau}} (x) \lambda \abs{x}^\beta = 1\), we conclude that \(\lim_{\abs{x} \to \infty} \lambda u(x) \abs{x}^\beta = 1\).
\end{proof}

The proof yields in fact a stronger statement that are not needed in the sequel:
\[
  u (x) = \frac{1}{\lambda \abs{x}^\beta} + O\Bigl(\frac{1}{\abs{x}^{\beta + 2}}\Bigr),
\]
as \(\abs{x} \to \infty\).
By Phragmen--Lindel\"of theory, the assumption \(\beta > 0\) can be dropped and the assumption \(\lim_{\abs{x} \to \infty} u (x) = 0\) can be replaced by \(\lim_{\abs{x} \to \infty} \abs{x}^{\frac{N - 1}{2}} e^{-\sqrt{\lambda} x} u (x) = 0\).

\begin{proof}[Proof of proposition~\ref{propositionSublinearAsymptotics}]
By proposition~\ref{propositionRieszAsymptotics}, there exists \(\mu \in \R\) such that
we have for every \(x \in \R^N\)
\begin{equation}
\label{eqSublinearAsymptoticsRiesz}
 \Bigabs{I_\alpha \ast u^p (x) - I_\alpha(x) \int_{\R^N} \abs{u}^p}
  \le \frac{\mu}{\abs{x}^{N - \alpha + \delta}},
\end{equation}
with \(0 < \delta \le \min (1, N (p - 1))\).

By proposition~\ref{propositionRegularity}, \(u > 0\) and \(u \in C^2 (\R^N)\). Hence, by the chain rule, \(u^{2 - p} \in C^2 (\R^N)\) and
\[
 - \Delta u^{2 - p} = - (2 - p) u^{1 - p} \Delta u + (2 - p) (p - 1) \abs{\nabla u}^2
\]
in \(\R^N\).
Since \(p \in (1, 2)\), by the equation satisfied by \(u\) and by \eqref{eqSublinearAsymptoticsRiesz}, we have for every \(x \in \R^N\),
\[
  - \Delta u^{2 - p}  (x) + (2 - p) u^{2 - p} (x) \ge (2 - p) I_\alpha (x) \Bigl(\int_{\R^N} \abs{u}^p - \frac{\mu}{\abs{x}^{\delta}}\Bigr).
\]
Let \(\underline{u} : \R^N \setminus B_1\) be such that
\[
\left\{
\begin{aligned}
  -\Delta \underline{u} (x)  + (2 - p) \underline{u} (x) &=  (2 - p) I_\alpha (x) \Bigl(\int_{\R^N} \abs{u}^p - \frac{\mu}{\abs{x}^{\delta}}\Bigr) & & \text{for \(x \in \R^N \setminus \Bar{B}_1\),}\\
  \underline{u} (x) & = u (x)^{2 - p}  & & \text{for \(x \in \partial B_1\)},\\
  \lim_{\abs{x} \to \infty} \underline{u} (x) & = 0.
\end{aligned}
\right.
\]
By the comparison principle,
\(\underline{u} \le u^{2 - p} \) in \(\R^N \setminus B_1\).
We apply now lemma~\ref{lemmaLinearPowerAsymptotics} twice and use the linearity of the operator \(-\Delta + 1\) to write
\[
 \lim_{\abs{x} \to \infty}  \frac{\underline{u}}{I_\alpha (x)} = \int_{\R^N} \abs{u}^p.
\]
We conclude that
\begin{equation}
\label{eqLimInfSublinearAsymptotics}
 \liminf_{\abs{x} \to \infty}  \frac{\bigl(u (x)\bigr)^{2 - p}}{I_\alpha (x)} \ge \int_{\R^N} \abs{u}^p.
\end{equation}

For the converse inequality, we note that by Young's inequality,
\[
  (I_\alpha \ast \abs{u}^p) \abs{u}^{p - 2} u \le (2 - p) (I_\alpha \ast \abs{u}^p)^\frac{1}{2 - p} + (p - 1) u.
\]
By \eqref{eqSublinearAsymptoticsRiesz}, we have for every \(x \in \R^N \setminus B_1\).
\[
\begin{split}
 \bigl((I_\alpha \ast \abs{u}^p) (x)\bigr)^\frac{1}{2 - p}
&\le I_\alpha (x)^\frac{1}{2 - p} \Bigl(\int_{\R^N} \abs{u}^p + \frac{C}{\abs{x}^\delta}\Bigr)^\frac{1}{p - 2}\\
&\le I_\alpha (x)^\frac{1}{2 - p} \biggl(\Bigl(\int_{\R^N} \abs{u}^p\Bigr)^\frac{1}{2 - p} + \frac{\nu}{\abs{x}^\delta}\biggr).
\end{split}
\]
Therefore, for every \(x \in \R^N \setminus B_1\).
\[
  -\Delta u (x)  + (2 - p) u (x) \le (2 - p) I_\alpha (x)^\frac{1}{2 - p} \biggl(\Bigl(\int_{\R^N} \abs{u}^p\Bigr)^\frac{1}{2 - p} + \frac{\nu}{\abs{x}^\delta}\biggr).
\]
Define now \(\Bar{u} \in C^2 (\R^N \setminus B_1)\) by
\[
\left\{
\begin{aligned}
  -\Delta \Bar{u} (x)  + (2\! -\! p)\Bar{u} (x) &= (2\! -\! p) I_\alpha (x)^\frac{1}{2 - p} \biggl(\!\Bigl(\int_{\R^N} \abs{u}^p\Bigr)^\frac{1}{2 - p}\! +\! \frac{\nu}{\abs{x}^\delta}\!\Biggr) & & \text{if \(x \in \R^N\!\setminus \Bar{B}_1\),}\\
  \Bar{u} (x) & = u (x)  & & \text{if \(x \in \partial B_1\)},\\
  \lim_{\abs{x} \to \infty} \Bar{u} (x) & = 0.
\end{aligned}
\right.
\]
By the comparison principle, we have \(u \le \Bar{u}\) in \(\R^N \setminus B_1\).
By lemma~\ref{lemmaLinearPowerAsymptotics}, we have
\[
 \lim_{\abs{x} \to \infty} \frac{\Bar{u}(x)}{I_\alpha (x)^\frac{1}{2 - p}}
= \Bigl( \int_{\R^N} \abs{u}^p \Bigr)^\frac{1}{2 - p}.
\]
Thus
\begin{equation}
\label{eqLimSupSublinearAsymptotics}
 \limsup_{\abs{x} \to \infty}  \frac{\bigl(u (x)\bigr)^{2 - p}}{I_\alpha (x)} \le \int_{\R^N} \abs{u}^p,
\end{equation}
and the assertion follows from the combination of \eqref{eqLimInfSublinearAsymptotics} and \eqref{eqLimSupSublinearAsymptotics}.
\end{proof}

\end{document}